\documentclass[12pt,leqno,twoside]{amsart}

\usepackage{latexsym,esint}

\theoremstyle{plain}

\def\endproof{\hspace*{\fill}\mbox{\ \rule{.1in}{.1in}}\medskip }

\newtheorem{theorem}{Theorem}[section]
\newtheorem{corollary}[theorem]{Corollary}
\newtheorem{lemma}[theorem]{Lemma}

\theoremstyle{definition}

\numberwithin{equation}{section}
\numberwithin{figure}{section}

\begin{document}

\title[Elastic plates with residual strain]
{The von K\'arm\'an equations \\ for plates with residual strain}
\author{Marta Lewicka, L. Mahadevan and Reza Pakzad}
\address{Marta Lewicka, University of Minnesota, Department of Mathematics, 
206 Church St. S.E., Minneapolis, MN 55455}
\address{L. Mahadevan, Harvard University, School of Engineering and Applied Sciences,
Cambridge, MA 02138}
\address{Reza Pakzad, University of Pittsburgh, 
Department of Mathematics, 139 University Place, Pittsburgh, PA 15260}
\email{lewicka@math.umn.edu, lm@seas.harvard.edu, pakzad@pitt.edu}
 
\label{firstpage}
 
\maketitle
\begin{abstract}
We provide a  derivation of the F\"oppl-von K\'arm\'an equations for
the shape of and stresses in an elastic plate with residual strains. These might
arise from a range of causes: inhomogeneous growth, plastic
deformation, swelling or shrinkage driven by solvent absorption. Our
analysis gives rigorous bounds on the convergence of the three dimensional
equations of elasticity to the low-dimensional description embodied in
the plate-like description of laminae and thus justifies a recent
formulation of the problem to the shape of growing leaves. It also
formalizes a procedure that can be used to derive other
low-dimensional descriptions of active materials.
\end{abstract}


\section{Introduction}

Laminae or leaf-like structures are thin, i.e. they have one dimension
much smaller than the other two. They arise in science and technology in a 
variety of situations, from atomically thin graphene (thickness $h \sim o(1) nm$) with a lateral span of a few cm, to the earth's crust ($h \sim 10 
km$) which spans thousands of km laterally. On the everyday scale, recently there has been much activity on trying to understand the mechanics of these
laminae when they are active or actuated, as in a growing leaf, a swelling or shrinking sheet of gel, a plastically strained sheet etc. In all these situations, 
the shape of the lamina arises as a consequence of the fact that inelastic effects associated with growth, swelling or shrinkage, plasticity, etc. results in 
a local and heterogeneous incompatibility of strains that leads to local elastic stresses. When combined with force balance, this naturally leads to the
non-trivial shapes that are seen even in the absence of any external forces. A simple experiment suffices to make this point - when growing leaves or 
plastically strained ribbons \cite{Sharon2007} are cut in different directions  to partially relieve the incompatible strains due to growth or plasticity, they 
relax to different shapes.

Recently, these observations have lead to a quest for a theory that describes the coupling between residual strain that might arise from 
a multitude of causes to the ultimate shape of the object. Given that the deformations and strains involved are not necessarily small, this raises an 
age-old question of how to decompose the deformations into the elastic and inelastic parts. A possible approach is to consider the equilibrated shape
of grown bodies and use a multiplicative decomposition of the deformation gradient that borrows from the theory of crystal plasticity which requires the 
notion of a reference configuration with respect to which all displacements are measured. This equilibrium approach leads to conceptual difficulties when 
matter is not conserved, as in growth processes, and suggests an incremental approach embodied in an evolutionary rather than an equilibrium process, 
which also has antecedents in the theory of plasticity. In either case, we then require the knowledge of a constitutive law in addition to a characterization 
of the geometry of the body. For soft materials such as plant and animal tissue, a reasonable assumption is that insofar as the elastic response is
concerned, the material is hyperelastic, while the inelastic deformations follow different laws depending on their origin. For example, swelling or shrinking
gels may be described by a poroelastic theory that couples fluid flow to osmotic stress and deformation, while growth in biological tissues arises from
cell proliferation or cell shape change that have their own description which couple molecular and macroscopic processes. For example, at the molecular 
level, mutants responsible for differential cell proliferation \cite{Nath2003} lead to a range of leaf shapes. At the macroscopic level, stresses induced by 
external loads lead to phenotypic plasticity in algal blades that switch between long, narrow blade-like shapes in rapid flow  to broader undulating shapes 
in slow flow \cite{Koehl2008}.  

Recent work  has focused on some of these questions by using variants of thin plate theory to highlight the self-similar structures that form near the edge 
due to variations in a prescribed  intrinsic metric  of a surface that is asymptotically flat at infinity \cite{Audoly2004}, and also on the case of a circular 
disk with edge-localized growth \cite{kupferman, Dervaux2008}, the shape of a long leaf \cite{Maha} etc. However, the theories used are 
not all identical and some of them arbitrarily ignore certain terms and boundary conditions without prior justification. This suggests that it might be useful
to rigorously derive an asymptotic theory for the shape of a residually strained thin lamina to clarify the role of the assumptions used while shedding light
on the errors associated with the use of the approximate theory that results.
Recently, such rigorous derivations were carried out \cite{FJMhier, lemopa1, lemopa3, LePa2} 
in the context of standard nonlinear elasticity for thin plates and shells.
Further, in \cite{LePa1}, a residually strained version of the Kirchhoff theory for plates 
\cite{kirchhoff} was rigorously derived under the assumption that the target metric is independant of thickness.  

In this paper, we carry out such a derivation 
under a different assumption on the  asymptotic behavior of the prescribed metric
and show that the resulting equations
are identical to those postulated to account for the effects of growth in elastic plates \cite{Maha} and used to describe the shape of a long leaf. We 
limit ourselves to the case when a decomposition of the deformation gradient into an elastic and inelastic part can be carried out - this requires that
it is possible to separate out a reference configuration, and is thus most relevant for the description of plant morphogenesis.
Although our results are valid for thin laminae that might be residually strained by a variety of means, here we will limit ourselves to studying the role of
incompatible growth strain in determining the shape of the laminae. In particular, we will only consider the one-way coupling of growth to shape and ignore 
the feedback from shape back to growth. 


%

\medskip

We now proceed to give an overview of the main results in the paper.

\medskip

{\bf 1. Geometry of incompatible strain.}
For a given mid-plate $\Omega$ which is an open bounded and simply connected
subset of $\mathbb{R}^2$, consider a sequence of $3$d plates:
$$\Omega^h = \Omega\times (-h/2, h/2), \qquad 0<h< <1,$$
viewed as the reference configurations of thin elastic (and homogeneous) 
tissues.
A typical point in $\Omega^h$ has the form $(x',hx_3)$ where 
$x'\in\Omega$ and $|x_3|<1$, and 
we shall make no distinction between
points $x'\in\Omega$ and $(x',0)\in \Omega^h$.

Each $\Omega^h$ is assumed to undergo a growth process, whose instantaneous
growth is described by a smooth tensor $a^h=[a_{ij}^h]
:\overline{\Omega^h}\longrightarrow\mathbb{R}^{3\times 3}$, 
with the property:
$$\forall x\in\overline{\Omega^h} \qquad \det a^h(x)>0.$$
In general, $a^h$ may follow its own dynamical evolution. However, here we will focus only on the effect of the tensor $a^h$ on the effective elastic theory for the grown body. Here we will use a multiplicative decomposition of the deformation gradient that is similar to the one used in plasticity and
also used in various growth formalisms, e.g. \cite{Rod}, $u :\Omega^h \longrightarrow \mathbb{R}^3$:
$$ \nabla u = F a^h.$$ 
The tensor $F= \nabla u (a^h)^{-1}$ corresponds to 
the elastic part of the deformation $u$, and accounts for the reorganization 
of the body $\Omega^h$ in response to the growth tensor $a^h$.  
The above assumes that it is possible to differentiate a reference configuration with respect to which one might measure all relative displacements. 
As mentioned before, this is true for some but certainly not all growth processes. In particular, this is reasonable for botanical growth processes,
but is unlikely for animal growth and remodeling processes that include large scale tissue flows and movements. 

\medskip

{\bf 2. Elastic energy associated with residual strains.}
The elastic energy of $u$ is now a function of $F$ only, and it is given by:
\begin{equation}\label{IhW}
I^h_W(u) = \frac{1}{h}\int_{\Omega^h} W(F) ~\mbox{d}x 
= \frac{1}{h}\int_{\Omega^h} W(\nabla u(a^h)^{-1}) ~\mbox{d}x,
\quad \forall u\in W^{1,2}(\Omega^h,\mathbb{R}^3).
\end{equation}
The elastic energy density $W:\mathbb{R}^{3\times 3}\longrightarrow \mathbb{R}_{+}$ 
is assumed to be compatible with the conditions of normalization 
and frame indifference (with respect to the special orthogonal group 
$SO(3)$ of proper rotations in $\mathbb{R}^3$):
\begin{equation*}\tag{$\clubsuit$i}
\forall F\in \mathbb{R}^{3\times 3} \quad
\forall R\in SO(3) \qquad
W(R) = 0, \quad W(RF) = W(F).
\end{equation*}
Further, we shall require the nondegeneracy of $W$ in the sense that:
\begin{equation*}\tag{$\clubsuit$ii}
\exists c>0\quad \forall F\in \mathbb{R}^{3\times 3}\qquad
W(F)\geq c~ \mathrm{dist}^2(F, SO(3)),
\end{equation*}
and also assume that $W$ is $\mathcal{C}^2$ regular in a neighborhood of $SO(3)$. 

We note that by the polar decomposition theorem every $a^h$ can be uniquely 
written as a product $R^h\tilde a^h$,
where $R^h \in SO(3)$ and $\tilde a^h$ is symmetric positive definite. Hence,
if $W$ is moreover isotropic:
\begin{equation}\label{isotropic}
\forall F\in \mathbb{R}^{3\times 3} \quad
\forall R\in SO(3) \qquad W(FR) = W(F),
\end{equation}
then without loss of generality we can assume $a^h$ to be symmetric
positive definite.

As we shall see, it is instructive to study the following energy functional:
\begin{equation}\label{I0}
I^h_0(u) = \frac{1}{h}
\int_{\Omega^h}\mbox{dist}^2\left(\nabla u(x) (a^h(x))^{-1},
SO(3)\right)~\mbox{d}x,
\end{equation}
as the energy in (\ref{IhW}) obeys a bound from below: $I^h_W\geq c I^h_0$.  

\medskip

{\bf 3. Relation with the non-Euclidean elasticity.}
We shall here compare the above approach with the  {\it target metric}
formalism proposed  in \cite{kupferman} and further developed in 
\cite{LePa1}.
On each $\Omega^h$ we assume that we are given a smooth Riemannian metric
$g^h=[g_{ij}^h]$.
The matrix fields $g^h:\overline{\Omega^h}\longrightarrow\mathbb{R}^{3\times 3}$
are therefore symmetric and strictly positive definite up to the boundary
$\partial \Omega^h$.
Let $\sqrt{g^h}$ be the unique symmetric positive definite square root
of $g^h$ and define, for all $x\in\overline{\Omega^h}$, the set:
\begin{equation}\label{Fh}
\mathcal{F}^h(x)=\left\{R\sqrt{g^h}(x); ~ R\in SO(3)\right\}\subset
\mathbb{R}^{3\times 3}.
\end{equation}
By the polar decomposition theorem, the necessary and sufficient condition 
for a deformation $u$ on $\Omega^h$
to be an orientation preserving realization of $g^h$:
$$(\nabla u)^T\nabla u = g^h \mbox{ and }~ \mbox{det}\nabla u>0 
\quad \mbox{ a.e.  in } \Omega^h$$
is the following:
$$\nabla u(x)\in \mathcal{F}^h(x) \quad \mbox{ a.e. in } \Omega^h.$$
Motivated by this observation, we can replace the energy functional (\ref{I0})
above by:
\begin{equation}\label{Iz0}
\tilde I^h_0(u) = \frac{1}{h}
\int_{\Omega^h}\mbox{dist}^2(\nabla u(x), \mathcal{F}^h(x))~\mbox{d}x
\qquad \forall u\in W^{1,2}(\Omega^h,\mathbb{R}^3),
\end{equation}
measuring (in $L^2$) the pointwise deviation of $u$
from being an orientation preserving realization of the given metric $g^h$.
We note that $\tilde I_0^h$ is comparable in magnitude with $I^h_0$.
Indeed, the intrinsic metric of the material is transformed 
by the growth tensor $a^h$ to the  target metric $g^h = (a^h)^T a^h$
and it is only the symmetric positive definite
part of $a^h$ given by $\sqrt{g^h} = \sqrt{(a^h)^T a^h}$ which plays the
role in determining the deformed shape of the material. 

\bigskip

Our main results can be divided in four major subcategories, 
presented below. 

\medskip

{\bf 4. Main results: Scaling analysis of thin non-Euclidean plates.} 
Given a sequence of growth tensors $a^h$, each close to $\mbox{Id}$ and
defined on $\Omega^h$, the main objective 
is to analyse the behavior of the minimizers of the corresponding energies $I^h_W$
as $h\to 0$.   
Let us recall that, as proved in \cite{LePa1}, the infimum:
$$ m_h = \inf\Big\{I^h_W(u); ~ u\in W^{1,2}(\Omega^h, \mathbb{R}^3)\Big\}$$
must be strictly positive whenever the Riemann curvature tensor 
of the metric $g^h = (a^h)^T a^h $ does not vanish identically on $\Omega^h$.
This condition for $g^h$, under suitable scaling properties, can be translated into a 
first order curvature condition (\ref{lincurvature}) below. In a first step (Theorem \ref {contra_general}) 
we establish a lower bound on $m_h$ in terms of a power law:
$$ m_h \geq c h^\beta, $$ 
for all values of $\beta$ greater than a critical $\beta_0$ in (\ref{critical}). 
This critical exponent  depends on the asymptotic 
behavior of the perturbation $a^h-\mathrm{Id}$ in terms of the thickness $h$. 

Under existence conditions for certain classes of isometries, 
it can be established that actually $m_h \sim h^{\beta_0}$. 
In other words, our analysis includes identification of the magnitude 
of the elastic energy of minimizers of $I^h_W$, in terms of the thicknesss $h$. 

\medskip

The following quantity measures the essential 
variation of the tensors $a^h$:
$$Var(a^h) = \|\nabla_{tan} (a^h_{~|\Omega})\|_{L^\infty(\Omega)} 
+ \|\partial_3 a^h\|_{L^\infty(\Omega^h)} $$
together with their scaling in $h$:
$$\omega_1 = \sup\left\{\omega;~ \lim_{h\to 0}\frac{1}{h^\omega}
Var(a^h) = 0\right\}.$$
The symbol $\nabla_{tan}$ denotes taking derivatives $\partial_1$ and $\partial_2$
in the in-plate directions $e_1=(1,0,0)^T$ and $e_2=(0,1,0)^T$. 
The derivative $\partial_3$ is taken in the out-of-plate direction $e_3=(0,0,1)^T$.
We will work under the following hypothesis:
\begin{equation*}\tag{$\spadesuit$i}
\|a^h\|_{L^\infty(\Omega^h)} + \|(a^h)^{-1}\|_{L^\infty(\Omega^h)} \leq C,
\end{equation*}
\begin{equation*}\tag{$\spadesuit$ii}
\omega_1 > 0
\end{equation*} 
 
\begin{theorem}\label{contra_general}
Assume $(\spadesuit)$. Assume that for some $\omega_0\geq 0$, 
there exists the limit:
\begin{equation*}\label{Eg}
\epsilon_g (x') = \lim_{h\to 0} \frac{1}{h^{\omega_0}}
\fint_{-h/2}^{h/2}a^h(x',t) - \mathrm{Id}~\mathrm{d}t
\quad \mbox{ in } L^2(\Omega, \mathbb{R}^{3\times 3}).
\end{equation*}
which moreover satisfies:
\begin{equation}\label{lincurvature} 
\mathrm{curl}^T\mathrm{curl}~(\epsilon_g)_{2\times 2} \not\equiv 0
\end{equation} 
and that $\omega_0<\min\{2\omega_1,\omega_1 + 1\}$. Then, for every
$\beta$ with: 
\begin{equation}\label{critical} 
\beta> \beta_0= \max\{\omega_0+2, 2\omega_0\},
\end{equation} we have:
$\displaystyle{\limsup_{h\to 0}\frac{1}{h^\beta}\inf I_0^h = +\infty}$.
\end{theorem}
Above, we used the following notational convention which will be employed
throughout the paper. For a matrix $F$, its $n\times m$ principle 
minor is denoted by $F_{n\times m}$. The superscript $^T$ refers to the 
transpose of a matrix or an operator. The operator $\mbox{curl}^T\mbox{curl}$ 
acts on $2\times 2$ square matrix fields $F$
by taking first $\mbox{curl}$ of each row (returning $2$ scalars)
and then taking $\mbox{curl}$ of the resulting $2$d vector, so that:
$\mbox{curl}^T\mbox{curl} F = \partial_{11}^2 F_{22} - \partial_{12}^2(F_{12}+F_{21})
+ \partial_{22}^2 F_{11}.$
The symmetric part of a square matrix $F$ is denoted by $\mbox{sym } F
= 1/2(F + F^T)$. In particular, we readily see that:
$\mbox{curl}^T\mbox{curl } F = \mbox{curl}^T\mbox{curl} (\mbox{sym }F)$. Physically this condition corresponds to the fact that the
growth strain $\epsilon_g$ is incompatible, i.e. it is not uniquely integrable and thus is not derivable from  an elastic deformation gradient.

\medskip

{\bf 5. Main results: Compactness.}  
By a compactness result we mean identification of the limit behavior 
of the minimizing sequence to $I^h_W$. More generally, this analysis can also 
be done for any sequence of deformations $u^h\in W^{1,2}(\Omega^h,\mathbb R^3)$ 
whose energy $I^h_W(u^h)$ scale like $h^{\beta_0}$. 
In the scaling regimes considered in this paper this compactness result 
has the following form: first, modulo rigid motions 
the deformations $u^h$ converge, up to a subsequence and in a suitable space, 
to the identity map on $\Omega$. 
Second, the suitably re-scaled displacement fields converge to elements of 
certain classes of Sobolev infinitesimal isometries. 

Note that no assumptions will be made on the special form of 
the deformations $u^h$. From this point of view our analysis is {\it Ansatz-free} 
and the limiting behavior of minimizers is rigorously shown to 
depend only on the choice of the sequence $a^h$.    

We present the compactness result (Theorem \ref{compactness} below)  
assuming the special form of the growth 
tensor (\ref{ahform}) which corresonds to the von K\'arm\'an model where 
$\beta_0 =4$. In this context, the out-of-plane displacement $v$ will 
be rescaled by the thickness $h$, and the in-plane displacement 
$w$ by $h^2$. Again, these scalings are naturally 
imposed by the original choice of the growth tensor $a^h$.

\begin{theorem}\label{compactness} 
Given two smooth matrix fields
$\epsilon_g,\kappa_g:\overline\Omega\longrightarrow \mathbb{R}^{3\times 3}$, 
define the growth tensors as:
\begin{equation}\label{ahform}
a^h(x', x_3)=\mathrm{Id} + h^2\epsilon_g(x') + hx_3\kappa_g(x').
\end{equation}
Assume that the energies of a sequence of deformations 
$u^h\in W^{1,2}(\Omega^h,\mathbb{R}^3$) satisfy: 
\begin{equation} \label{boundh4} 
I^h_W(u^h) \leq Ch^4,
\end{equation} 
where $W$ fulfills $(\clubsuit)$. Then there exist 
proper rotations $\bar R^h\in SO(3)$ and translations 
$c^h\in\mathbb{R}^3$ such that for the normalized deformations:
$$y^h(x',x_3) = (\bar R^h)^T u^h(x',hx_3) - c^h:\Omega^1\longrightarrow\mathbb{R}^3$$
the following holds.
\begin{itemize}
\item[(i)] $y^h(x',x_3)$ converge in $W^{1,2}(\Omega^1,\mathbb{R}^3)$ to $x'$.
\item[(ii)] The scaled displacements:
\begin{equation}\label{Vh}
V^h(x')=\frac{1}{h}\fint_{-1/2}^{1/2}y^h(x',t) - x'~\mathrm{d}t
\end{equation}
converge (up to a subsequence) in $W^{1,2}(\Omega,\mathbb{R}^3)$ to the vector field
of the form $(0,0,v)^T$, with the only non-zero out-of-plane
scalar component: $v\in W^{2,2}(\Omega,\mathbb{R})$.
\item[(iii)] The scaled in-plane displacements $h^{-1} V^h_{tan}$ converge
(up to a subsequence) weakly in $W^{1,2}(\Omega, \mathbb{R}^2)$ 
to an in-plane displacement field $w\in W^{1,2}(\Omega,\mathbb{R}^2)$.
\end{itemize} 
\end{theorem} 

\medskip

{\bf 6. Main results: $\Gamma$-convergence.} 
Heuristically, a sequence of functionals $F_n$ is said to $\Gamma$ converge 
to a limit functional $F$ if the the minimizers of $F_n$, if converging, 
have a minimizer of $F$ as a limit. 
More precisely, any $\Gamma$-convergence result involves a
careful comparison of the values of the energies $F_n$ on seqeunces 
of deformations and the value of $F$ on the limit deformation. 
Hence, it combines a lower and an upper bound estimate (which are called 
the $\Gamma$-liminf and the $\Gamma$-limsup inequalities). 
 
For the von K\'arm\'an growth tensor studied in this article, these estimates 
are established for the sequence $ 1 /h^4 I^h_W(u^h)$ and the limit energy 
value ${\mathcal I}_g(w,v)$ given in (\ref{vonKarman}) below.  
The liminf inequality (Theorem \ref{liminf}) involves a lower bound 
on the energy of any sequence of deformations $u^h$. 
The limsup part (Theorem \ref{thmaindue}) establishes that for any pair 
of displacements $(w,v)$ in suitable limit spaces, one can construct a sequence 
(\ref{recoveryseq}) of $3$d deformations of thin plates $\Omega^h$   
which approximately yield the energy $\mathcal{I}_g(w,v)$.  
The form of such {\it recovery sequence} delivers an insight on how to reconstruct 
the $3$d deformations out of the data on the mid-plate $\Omega$. 
In particular, comparing the present von K\'arm\'an growth 
model (\ref{recoveryseq}) with the classical model (\cite{FJMhier}, Section 6.1) 
we observe the novel warping effect in the non-tangential growth (see (\ref{d01})). 

\begin{theorem}\label{liminf} 
Assume (\ref{ahform}) and $(\clubsuit)$. Let the bound (\ref{boundh4})
be satisfied by a sequence $u^h\in W^{1,2}(\Omega^h,\mathbb{R}^3)$ so that 
the convergences (i), (ii) and (iii) of Theorem \ref{compactness} hold true. 
Then there holds:
$$\liminf_{h\to 0} \frac{1}{h^4} I_W^h(u^h) \geq \mathcal{I}_g(w,v),$$
where:
\begin{equation}\label{vonKarman}
\begin{split}
\mathcal{I}_g(w,v)= 
\frac{1}{2} 
\int_\Omega \mathcal{Q}_2&\left(\mathrm{sym }\nabla w 
+\frac{1}{2}\nabla v\otimes \nabla v
- (\mathrm{sym}~ \epsilon_g)_{2\times 2}\right)\\
&\qquad\qquad\qquad
 + \frac{1}{24} \int_\Omega \mathcal{Q}_2\Big(\nabla^2 v 
+ (\mathrm{sym}~ \kappa_g)_{2\times 2}\Big),
\end{split}
\end{equation}
and the quadratic nondegenerate form $\mathcal{Q}_2$, acting on matrices 
$F\in\mathbb{R}^{2\times 2}$ is:
\begin{equation}\label{defQ}
\mathcal{Q}_2(F) =\min\{\mathcal{Q}_3(\tilde F); ~\tilde F\in\mathbb{R}^{3\times 3},
\tilde F_{2\times 2}= F\} \mbox{ and } \mathcal{Q}_3(\tilde F) =
D^2W(\mathrm{Id})(\tilde F\otimes \tilde F).
\end{equation}
\end{theorem}

\begin{theorem}\label{thmaindue}
Assume (\ref{ahform}) and $(\clubsuit)$.
Then, for every $w\in W^{1,2}(\Omega,\mathbb{R}^3)$ and every 
$v\in W^{2,2}(\Omega,\mathbb{R})$, there exists a sequence of deformations 
$u^h\in W^{1,2}(S^{h},\mathbb{R}^3)$ such that the following holds:
\begin{itemize}
\item[(i)] The sequence 
$y^h(x',x_3) = u^h(x',hx_3)$ converge in $W^{1,2}(\Omega^1,\mathbb{R}^3)$ to $x'$.
\item[(ii)] $\displaystyle V^h(x') = h^{-1}\fint_{-h/2}^{h/2}(u^h(x',t) 
- x')~\mathrm{d}t$ 
converge in $W^{1,2}(\Omega,\mathbb{R}^3)$ to $(0,0,v)^T$.
\item[(iii)] $h^{-1} V^h_{tan}$ converge in $W^{1,2}(\Omega,\mathbb{R}^{2})$ to $w$.
\item[(iv)] Recalling the definition (\ref{vonKarman}) one has:
$$\lim_{h\to 0} \frac{1}{h^4} I_W^h(u^h) = \mathcal{I}_g(w,v).$$
\end{itemize}
\end{theorem}

The main consequence of the $\Gamma$-convergence result is the following: 
If $u^h$ is a minimizing sequence for $I^h_W$, and if $w$ and $v$ 
are the respective limiting in-plane and out-of-plane displacements 
corresponding to $u^h$, then $(w,u)$  will be a minimizer of the 
von K\'arm\'an growth functional ${\mathcal I}_g$ (Corollary 
\ref{minsconverge} below).

Another direct corollary is the identification 
of the lower bound of $I^h_W$ for the von K\'arm\'an growth (\ref{ahform}) 
under an appropriate curvature condition. 
Note that the assumptions of Theorem \ref{contra_general} do not hold,
since in the present case $\omega_0 = 2\omega_1 = \omega_1 +1 = 2$. 
However, if we replace (\ref{lincurvature}) by either of the following 
two conditions: 
\begin{equation}\label{CO1}
\mathrm{curl }\big((\mathrm{sym }~\kappa_g)_{2\times 2}\big) \neq 0, 
\end{equation}
or:
\begin{equation}\label{CO2} 
\displaystyle\mathrm{curl}^T\mathrm{curl}~ (\epsilon_g)_{2\times 2} 
+ \mathrm{det}\big((\mathrm{sym }~\kappa_g)_{2\times 2}\big) \neq 0,
\end{equation}
we indeed obtain that $ \inf I^h_W \ge ch^4$ with $c>0$.  
The above conditions guarantee that the highest order terms in 
the expansion of the Riemann curvature tensor components 
$R_{1213}$, $R_{2321}$ and $R_{1212}$ of $g^h=(a^h)^Ta^h$ do not vanish. 
Also, either of them implies that $\inf \mathcal {I}_g >0$ (see Lemma \ref{GCM}), 
which combined with Theorem \ref{liminf} yields the lower bound on $\inf I^h_W$.  

\begin{corollary}\label{minsconverge}
Assume (\ref{ahform}) and $(\clubsuit)$. Then:
\begin{itemize}
\item[(i)] There exist uniform constants $C, c\geq 0 $ such that for every 
$h$ there holds:
\begin{equation}\label{bounds} 
c\le \frac 1{h^4} \inf I^h_W \le C.  
\end{equation}
If moreover (\ref{CO1}) or (\ref{CO2}) hold then one may have $c>0$.
\item[(ii)] There exists at least one minimizing sequence 
$u^h\in W^{1,2}(\Omega^h,\mathbb{R}^3$ for $I^h_W$:
\begin{equation}\label{approxmin}
\lim_{h\to 0} \Big ( \frac 1{h^4} I^h_W (u^h) - \frac 1{h^4} \inf I^h_W \Big ) = 0.  
\end{equation} 
For any such sequence the convergences (i), (ii) and (iii) 
of Theorem \ref{compactness} hold and the limit $(w,v)$ is a minimizer 
of $\mathcal I_g$.  
\item[(iii)] For any minimizer $(w,v)$ of ${\mathcal I}_g$, there exists 
a minimizing sequence $u^h$, satisfying (\ref{approxmin}) together with
(i), (ii), (iii) and (iv) of Theorem \ref{thmaindue}.
\end{itemize}
\end{corollary} 

\medskip

{\bf 7. Main results: Euler-Lagrange equations for the limit theory.} 
When $W$ is frame invariant as in $(\clubsuit i)$ and isotropic (\ref{isotropic}),
one can see \cite{FJMhier} that the quadratic forms of (\ref{defQ})
are given explicitely as:
\begin{equation}\label{Q23}
\mathcal{Q}_3(F) = 2\mu |\mbox{sym } F|^2 + \lambda|\mbox{tr } F|^2,
\qquad \mathcal{Q}_2(F_{2\times 2}) = 2\mu |\mbox{sym } F_{2\times 2}|^2 + 
\frac{2\mu\lambda}{2\mu+\lambda}|\mbox{tr } F_{2\times 2}|^2,
\end{equation}
for all $F\in \mathbb{R}^{3\times 3}$, where $tr$ stands for the trace 
of a quadratic matrix, and $\mu$ and $\lambda$ are the Lam\'e constants, 
satisfying: $\mu\geq 0$, $3\lambda+\mu\geq 0$. 

We will show that under these conditions, the Euler-Lagrange equations 
(\ref{el1}), (\ref{el4}) of ${\mathcal I}_g$ in (\ref{vonKarman})
are equivalent, under a change of variables which replaces the in-plane 
displacement $w$ by the Airy stress potential $\Phi$, 
to the system of von K\'arm\'an-like equations introduced recently 
by Liang and Mahadevan in \cite{Maha}: 
\begin{equation}\label{EL-2}
\left \{ \begin{split}
\Delta^2\Phi & = -S(K_G + \lambda_g)\\
B\Delta^2v &= [v,\Phi] - B\Omega_g~,
\end{split}\right.
\end{equation} 
under the corresponding boundary conditions (\ref{bc1})  (\ref{b1})  (\ref{b2}),
where: 
\begin{equation*}
\left \{ \begin{split}
S & = \mbox{ Young's modulus } = \frac{\mu(3\lambda + 2\mu)}{\lambda+\mu},
\qquad K_G = \mbox{ Gaussian curvature } = \frac{1}{2}[v,v], \\
B & = \mbox{ bending stiffness } = \frac{S}{12(1-\nu^2)}, \qquad
\nu = \mbox{ Poisson's ratio } = \frac{\lambda}{2(\lambda + \mu)}, \\
\lambda_g & = \mbox{curl}^T\mbox{curl }(\epsilon_g)_{2\times 2}
= \partial_{22}(\epsilon_g)_{11} +  \partial_{11}(\epsilon_g)_{22}
- \partial_{12}\Big((\epsilon_g)_{12}+ (\epsilon_g)_{21}\Big),\\
\Omega_g & = \mbox{div}^T\mbox{div }\Big((\kappa_g)_{2\times 2}
+\nu\mbox{ cof }(\kappa_g)_{2\times 2}\Big) \\
& = \partial_{11} \Big((\kappa_g)_{11}+ \nu(\kappa_g)_{22}\Big) +
\partial_{22}\Big((\kappa_g)_{22}+ \nu(\kappa_g)_{11}\Big) +
(1- \nu)\partial_{12}\Big((\kappa_g)_{12}+ (\kappa_g)_{21}\Big).
\end{split}
\right.
\end{equation*}

\medskip

In brief, we may summarize our results as follows. 
Under the special form of the growth tensor given by (\ref{ahform}), we have: 
\begin{itemize} 
\item[(i)]$m_h=\inf I^h_W$ scales like $h^4$,
\item[(ii)] there exists a sequence of deformations $u^h$ such that 
$I^h_W(u^h) - m_h = o(h^4)$, and for which the rescaled in-plane and out-of-plane 
displacements converge to a limit $(w,v)$, 
\item[(iii)] $(w,v)$ minimizes the von K\'arm\'an growth functional
${\mathcal I}_g$ in (\ref{vonKarman}) and hence satisfies the system 
(\ref{EL-2}) with the corresponding free boundary conditions 
(\ref{bc1})  (\ref{b1})  (\ref{b2}), 
when expressed in terms of the Airy stress potential. 
\end{itemize}
 
\medskip

{\bf 8. Approximating low energy deformations.} 
A crucial step in obtaining the above results is an approximation theorem 
(Theorem \ref{thapprox} below). 
The underlying idea is that we can control the oscillations of 
the deformation gradient $\nabla u^h$ in boxes of diameter proportional 
to the thickness $h$. The geometric rigidty estimate of Friesecke, 
James and M\"uller \cite{FJMgeo} and its generalization 
to our setting 
is the basic tool in this step. 
\begin{theorem}\label{thapprox}
Assume $(\spadesuit)$. Let 
$u^h\in W^{1,2}(\Omega^h,\mathbb{R}^3)$ satisfy:
$$\lim_{h\to 0} \frac{1}{h^2} I_0^h(u^h) = 0.$$ 
Then there exist matrix fields $R^h\in W^{1,2}(\Omega,\mathbb{R}^{3\times 3})$,
such that $R^h(x')\in SO(3)$ for a.a. $x'\in\Omega$ and:
\begin{equation*}
\frac{1}{h}\int_{\Omega^h}|\nabla u^h(x) - R^h(x')a^h(x)|^2~\mathrm{d}x
\leq C \left(I_0^h(u^h) + h^2 Var^2(a^h)\right),
\end{equation*}
\begin{equation*}
\int_\Omega |\nabla R^h|^2 \leq C h^{-2} \left(I_0^h(u^h) + h^2Var^2(a^h)\right),
\end{equation*}
where the constant $C$ is independent of $h$.
\end{theorem}

\section{Scaling analysis:
a proof of Theorem \ref{contra_general}}

{\bf 1.} Take $\beta>\max\{\omega_0+2, 2\omega_0\}$ and assume, by contradiction, 
that for some sequence $u^h\in W^{1,2}(\Omega^h,\mathbb{R}^3)$ there holds:
\begin{equation}\label{p1}
\frac{1}{h^\beta}I_0^h(u^h) \leq C.
\end{equation}
Since $\beta>2$, in virtue of Theorem \ref{thapprox} there exists 
the rotation-valued matrix fields $R^h\in W^{1,2}(\Omega, SO(3))$ 
approximating appropriately $\nabla u^h (a^h)^{-1}$.
Observe that:
\begin{equation}\label{p2}
\mbox{dist}^2(\fint_\Omega R^h, SO(3))\leq\fint_\Omega|R^h(x) - R^h(x_0)|^2~\mbox{d}x
\leq C\int_\Omega|\nabla R^h|^2 \rightarrow 0 \quad \mbox{as } h\to 0,
\end{equation}
by the second estimate in Theorem \ref{thapprox} and by $(\spadesuit)$.
We may hence, for small $h$, define the averaged rotations $\bar R^h\in SO(3)$ by
$\bar R^h = \mathbb{P}_{SO(3)}\fint_\Omega R^h.$

Define now two fields: $V^h\in W^{1,2}(\Omega,\mathbb{R}^3)$ and 
$A^h\in W^{1,2}(\Omega,\mathbb{R}^{3\times 3})$:
$$V^h(x') = \frac{1}{h^{\omega_0}}\fint_{-h/2}^{h/2}(\bar R^h)^T u^h(x', t) 
- x'~\mbox{d}t,$$
$$A^h(x') = \frac{1}{h^{\omega_0}}
\left((\bar R^h)^T R^h(x')\fint_{-h/2}^{h/2} a^h(x',t)~\mbox{d}t -\mbox{Id}\right).$$
Observe that:
\begin{equation*}\label{p25}
\begin{split}
\|\nabla V^h - A^h_{3\times 2}\|^2_{L^2(\Omega)} &\leq \frac{C}{h^{2\omega_0}}\int_\Omega
\left| \fint_{-h/2}^{h/2}R^h(x')a^h_{3\times 2}(x',t) 
- \nabla_{tan}u^h(x',t)~\mbox{d}t \right|^2
~\mbox{d}x'\\ & \leq \frac{C}{h^{2\omega_0+1}}
\int_{\Omega^h} |\nabla u^h(x) - R^h(x') a^h(x)|^2
~\mbox{d}x \\ & \leq C\left(\frac{1}{h^{2\omega_0}} I_0^h(u^h) + \frac{1}{h^{2\omega_0-2}}
Var^2(a^h)\right) \rightarrow 0 \quad \mbox{as } h\to 0,
\end{split}
\end{equation*}
by Theorem \ref{thapprox} and since $2\omega_0<\beta$ and $2\omega_0 - 2<2\omega_1$.

{\bf 2.} Notice that:
\begin{equation*}
\begin{split}
A^h(x') = & (\bar R^h)^T R^h(x') \epsilon_g(x') 
+ \frac{1}{h^{\omega_0}}\left((\bar R^h)^T R^h(x') - \mbox{Id}\right)\\ &
- (\bar R^h)^T R^h(x')\left(\epsilon_g(x') - \frac{1}{h^{\omega_0}}
\fint_{-h/2}^{h/2} a^h(x',t) - \mbox{Id}~\mbox{d}t\right).
\end{split}
\end{equation*}
Clearly, the third term above converges in $L^2(\Omega)$ to $0$, 
by the definition (\ref{Eg}). The first term converges to
$\epsilon_g$, as by (\ref{p2}) and Theorem \ref{thapprox}:
\begin{equation}\label{p3}
\begin{split}
\int_\Omega& |(\bar R^h)^TR^h -\mbox{Id}|^2 
\leq C \int_\Omega |R^h-\bar R^h|^2 \\
&\leq C\left(
\int_\Omega|R^h-\fint_\Omega R^h|^2 + \mbox{dist}^2(\fint_\Omega R^h, SO(3))\right)
\leq C \int_\Omega |\nabla R^h|^2 \rightarrow 0 \quad \mbox{as } h\to 0.
\end{split}
\end{equation}

To deal with the second term, in the expansion of $A^h$, recall that
$\bar R^h, R^h(x')\in SO(3)$ and so:
\begin{equation*}\label{4.55}
\frac{1}{h^{\omega_0}}\mbox{sym} \left((\bar R^h)^T R^h - \mbox{Id}\right)
=\frac{1}{h^{\omega_0}} \left((\bar R^h)^T R^h - \mbox{Id}\right)^T
\left((\bar R^h)^T R^h - \mbox{Id}\right).
\end{equation*}
Therefore, reasoning as in in (\ref{p3}):
\begin{equation*}\label{4.7}
\begin{split}
\frac{1}{h^{\omega_0}} \left\|\mbox{sym}\left((\bar R^h)^T R^h 
- \mbox{Id}\right)\right\|_{L^2(\Omega)}
&\leq \frac{1}{h^{\omega_0}} \|(\bar R^h)^T R^h - \mbox{Id}\|^2_{L^4(\Omega)}\\
&\leq \frac{C}{h^{\omega_0}} \|R^h - \bar R^h\|^2_{W^{1,2}(\Omega)}
\leq \frac{C}{h^{\omega_0}} \int_\Omega |\nabla R^h|^2 \\
& \leq C\left( \frac{1}{h^{\omega_0+2}} I_0^h(u^h) + \frac{1}{h^{\omega_0}}
Var^2(a^h)\right) \rightarrow 0 \quad \mbox{as } h\to 0,
\end{split}
\end{equation*}
by Theorem \ref{thapprox} and since $\omega_0+2<\beta$ and $\omega_0<2\omega_1$.

{\bf 3.}  Summarizing, the previous step yields:
$$\mbox{sym }\epsilon_g = \lim_{h\to 0} \mbox{sym }A^h \quad \mbox{in } 
L^2(\Omega,\mathbb{R}^{3\times 3}),$$
which by (\ref{p25}) implies:
$$(\mbox{sym }\epsilon_g)_{2\times 2} 
= \lim_{h\to 0} \mbox{sym}\nabla V^h \quad \mbox{in } 
L^2(\Omega,\mathbb{R}^{2\times 2}).$$
Consequently, $(\mbox{sym }\epsilon_g)_{2\times 2} = \mbox{sym}\nabla V$, 
for some $V\in W^{1,2}(\Omega,\mathbb{R}^3)$, and hence there must be: 
$\mbox{curl}^T\mbox{curl }(\epsilon_g)_{2\times 2} = 
\mbox{curl}^T\mbox{curl }(\mbox{sym }\epsilon_g)_{2\times 2}=0$ in $\Omega$.
This brings a contradiction with (\ref{lincurvature}) and hence ends the proof.
\endproof


\section{Lower bound for the Von K\'arm\'an scaling:
a proof of Theorem \ref{compactness} and Theorem \ref{liminf}}

Consider a specific form of the growth tensor:
$$ a^h(x', x_3)=\mbox{Id} + h^\gamma\epsilon_g(x') + h^\theta x_3\kappa_g(x')$$
with exponents $\gamma,\theta >0$ and the smooth perturbation moments:
$\epsilon_g,\kappa_g:\overline\Omega\longrightarrow \mathbb{R}^{3\times 3}$.
One easily sees that $(\spadesuit)$ holds here with:
$$Var(a^h) = h^\gamma\|\nabla\epsilon_g\|_{L^\infty} + h^\theta\|\kappa_g\|_{L^\infty},
\quad \omega_1=\min\{\gamma,\theta\},\quad \omega_0 = \gamma.$$
Also, the result in Theorem \ref{contra_general} holds provided that
$\mbox{curl}^T \mbox{curl }(\mbox{sym }\epsilon_g)_{tan}\not\equiv 0$ and:
\begin{equation}\label{gam}
\gamma<\min\{\theta+1, 2\theta\}\quad \mbox{ and } 
\quad \beta>\max\{\gamma+2,2\gamma\}.
\end{equation}
In what follows, we shall work with exponents $\gamma=2,\theta=1,\beta=4$
which are critical for both inequalities in (\ref{gam}).

\bigskip 

\noindent {\bf Proof of Theorem \ref{compactness}}

\noindent 
{\bf 1.} Let $R^h\in W^{1,2}(\Omega, SO(3))$ be the matrix fields
as in Theorem \ref{thapprox}:
\begin{equation}\label{m1}
\frac{1}{h}\int_{\Omega^h}|\nabla u^h - R^ha^h|^2\leq Ch^4, \qquad 
\int_\Omega|\nabla R^h|^2\leq Ch^2.
\end{equation}
Define the averaged rotations:
$\tilde R^h = \mathbb{P}_{SO(3)}\fint_\Omega R^h.$
These projections of $\fint R^h$ onto $SO(3)$ are well defined for small $h$ 
in virtue of:
$$ \mbox{dist}^2(\fint_\Omega R^h, SO(3))\leq \fint_\Omega |R^h(x) - R^h(x_0)|^2
\leq C\int_\Omega |\nabla R^h|^2\leq Ch^2.$$
Further:
\begin{equation}\label{5.60}
\int_\Omega |R^h - \tilde R^h|^2 \leq C\Big(\int_\Omega |R^h - \fint R^h|^2
+ \mbox{dist}^2(\fint R^h, SO(3))\Big)\leq Ch^2
\end{equation}
Let now: 
\begin{equation}\label{m00}
\hat R^h = \mathbb{P}_{SO(3)}\fint_{\Omega^h} (\tilde R^h)^T\nabla u^h.
\end{equation}
The above projection is well defined for small $h$, because
$\mbox{dist}^2(\fint_{\Omega^h}(\tilde R^h)^T\nabla u^h, SO(3))$ is bounded by:
\begin{equation}\label{5.77}
\begin{split}
|\fint_{\Omega^h} &(\tilde R^h)^T\nabla u^h - \mbox{Id}|^2\leq
C\fint_{\Omega^h}|\nabla u^h - \tilde R^h|^2\\
&\leq C\Big(\fint_{\Omega^h} |\nabla u^h - R^ha^h|^2 + 
\fint_{\Omega^h} |a^h - \mbox{Id}|^2 + \fint_{\Omega^h}|R^h-\tilde R^h|^2\Big)
\leq Ch^2,
\end{split}
\end{equation}
where we used (\ref{m1}) and (\ref{5.60}).
Consequently, we also obtain:
\begin{equation}\label{5.66}
|\hat R^h -\mbox{Id}|^2\leq C|\mbox{skew}\fint_{\Omega^h}(\tilde R^h)^T\nabla u^h|^2
\leq C|\fint_{\Omega^h} (\tilde R^h)^T\nabla u^h - \mbox{Id}|^2\leq Ch^2.
\end{equation}
The first inequality above follows by noticing that for any matrix 
$F$ close to $\mbox{Id}$ there holds: 
$\mathbb{P}_{SO(3)}(\mbox{sym } F) = \mbox{Id}$, and hence:
$|\mathbb{P}_{SO(3)}F -\mbox{Id}|\leq C |F - \mbox{sym }F\leq C |\mbox{skew }F|$.

{\bf 2.} We may now define:
\begin{equation}\label{m02}
\bar R^h = \tilde R^h \hat R^h.
\end{equation}
By (\ref{5.60}),  (\ref{5.66}) and  (\ref{m1}) it follows that:
\begin{equation}\label{m2}
\int_\Omega |R^h - \bar R^h|^2 \leq Ch^2 
\quad\mbox{ and } \quad\lim_{h\to 0} (\bar R^h)^TR^h =\mbox{Id}
\quad \mbox{ in } W^{1,2}(\Omega, \mathbb{R}^{3\times 3}).
\end{equation}
Let $c^h\in\mathbb{R}^3$ be vectors such that for the rescaled 
averaged displacement $V^h$ defined as in 
(\ref{Vh}): $V^h(x') = h^{-1}\fint_{-h/2}^{h/2} (\bar R^h)^T u^h(x',t) 
- c^h - x'~\mbox{d}t$, there holds:
\begin{equation}\label{m22}
\int_\Omega V^h = 0, \qquad \mbox{skew}\int_\Omega\nabla V^h = 0.
\end{equation}
The second statement in (\ref{m22}) follows by noticing that, for a matrix $F$
sufficiently close $SO(3)$, its projection $R=\mathbb{P}_{SO(3)}F$ is coincides
with the unique rotation appearing in the polar decomposition of $F$, that is:
$F=RU$ with $\mbox{skew }U = 0$.
Therefore and in view of (\ref{m00}) and (\ref{m02}) we obtain that 
$(\bar R^h)^T \fint_{\Omega^h} \nabla u^h = 
(\hat R^h)^T \fint_{\Omega^h}(\tilde R^h)^T\nabla u^h$ is symmetric.
Hence:
$\mbox{skew}\fint_\Omega\nabla V^h$ $ = h^{-1}\mbox{skew}\fint_{\Omega^h}
(\bar R^h)^T\nabla u^h $ $= 0$.
In particular, we see as well that (\ref{m02}) coincides
with:
$$\bar R^h = \mathbb{P}_{SO(3)}\fint_{\Omega^h}\nabla u^h. $$
To obtain (i) we use (\ref{m02}), (\ref{5.77}) and (\ref{5.66}):
\begin{equation}\label{5.111}
\begin{split}
\|(\nabla y^h - &\mbox{Id})_{3\times 2}\|^2_{L^2(\Omega^1)} 
\leq \frac{1}{h}\int_{\Omega^h} |(\bar R^h)^T\nabla u^h - \mbox{Id}|^2 \\
&\leq C \left(\frac{1}{h}\int_{\Omega^h}
|(\tilde R^h)^T\nabla u^h - \mbox{Id}|^2 ~\mbox{d}x + |\hat R^h - \mbox{Id}|^2
\right) \leq Ch^2,
\end{split}
\end{equation}
and notice that by (\ref{m1}):
$$\|\partial_3 y^h\|^2_{L^2(\Omega^1)} \leq Ch\int_{\Omega^h}|\nabla u^h|^2
\leq Ch^2,$$
which implies convergence of $y^h$ by means of the Poincar\'e inequality 
$\int_{\Omega^1}y^h(x) - x' = h\int_\Omega V^h=0$ by (\ref{m22}).
Notice also that (\ref{5.111}) implies the weak convergence
(up to a subsequence) in $W^{1,2}(\Omega,\mathbb{R}^3)$ of $V^h$.

{\bf 3.} Consider the matrix fields $A^h\in W^{1,2}(\Omega,\mathbb{R}^{3\times 3})$:
\begin{equation*}
\begin{split}
A^h(x') &= \frac{1}{h}\fint_{-h/2}^{h/2} (\bar R^h)^TR^h(x') a^h(x',t) - \mbox{Id}
~\mbox{d}t\\ &  = h(\bar R^h)^TR^h(x')\epsilon_g(x') 
+ \frac{1}{h}\left((\bar R^h)^TR^h(x') - \mbox{Id}\right).
\end{split}
\end{equation*}
By (\ref{m2}) and (\ref{m1}), clearly: $\|A^h\|_{W^{1,2}(\Omega)}\leq C$ and so, 
up to a subsequence:
\begin{equation}\label{m3}
\begin{split}
\lim_{h\to 0} & A^h = A \quad  \mbox{ and } 
\quad  \lim_{h\to 0} \frac{1}{h} \left((\bar R^h)^TR^h- \mbox{Id}\right) = A \\
& \mbox{ weakly in } W^{1,2}(\Omega, \mathbb{R}^{3\times 3}) 
\mbox{ and (strongly) in } L^q(\Omega, \mathbb{R}^{3\times 3}) \quad \forall q\geq 1.
\end{split}
\end{equation}
Also, using (\ref{m2}) and (\ref{m1}) again: 
\begin{equation*}
\begin{split}
h^{-1}\|\mbox{sym}((\bar R^h)^TR^h &- \mbox{Id})\|_{L^2(\Omega)}
 = 2h^{-1}\|((\bar R^h)^TR^h - \mbox{Id})^T 
((\bar R^h)^TR^h - \mbox{Id})\|_{L^2(\Omega)} \\ & \leq Ch^{-1}
\|(\bar R^h)^TR^h - \mbox{Id}\|^2_{L^4(\Omega)} \leq Ch^{-1} 
\|R^h - \bar R^h\|^2_{W^{1,2}(\Omega)} \leq Ch.
\end{split}
\end{equation*}
Above, we used a straightforward observation that:
\begin{equation}\label{helpma}
(R-\mbox{Id})^T(R-\mbox{Id}) = -2\mbox{sym}(R-\mbox{Id})\qquad\forall R\in SO(3).
\end{equation}
Therefore, the limiting matrix field $A$ has skew values:
\begin{equation}\label{m4}
\mbox{sym } A = \lim_{h\to 0} \mbox{sym } A^h = 0.
\end{equation}
Further, by (\ref{helpma}) we observe that:
$$\frac{1}{h}\mbox{sym } A^h = \mbox{sym} \left((\bar R^h)^TR^h \epsilon_g\right)
- \frac{1}{2}\frac{1}{h^2} \left((\bar R^h)^TR^h(x')- \mbox{Id}\right)^T
\left((\bar R^h)^TR^h(x')- \mbox{Id}\right)$$
Hence, by (\ref{m2}), (\ref{m3}) and (\ref{m4}): 
\begin{equation}\label{m5}
\begin{split}
\lim_{h\to 0} \frac{1}{h}\mbox{sym} A^h =  \mbox{sym } \epsilon_g 
- \frac{1}{2} A^TA = \mbox{sym } \epsilon_g + \frac{1}{2}A^2 \quad 
\mbox{ in } L^q(\Omega, \mathbb{R}^{3\times 3}) \quad \forall q\geq 1.
\end{split}
\end{equation}

{\bf 4.} Regarding convergence of $V^h$, we have:
\begin{equation}\label{m55}
\nabla V^h(x') = A^h_{3\times 2}(x') + \frac{1}{h}(\bar R^h)^T \fint_{-h/2}^{h/2}
R^h(x') a^h_{3\times 2}(x',t) - \nabla_{tan} u^h(x',t)~\mbox{d}t.
\end{equation}
Also:
\begin{equation*}
\begin{split}
\|\nabla V^h - A^h_{3\times 2}\|_{L^2(\Omega)}^2 &
\leq \frac{C}{h^{2}}\int_\Omega\left| \fint_{-h/2}^{h/2}R^h(x')a^h_{3\times 2}(x',t) 
- \nabla_{tan}u^h(x',t)~\mbox{d}t \right|^2~\mbox{d}x'\\ &
\leq \frac{C}{h^{3}}\int_{\Omega^h} |\nabla u^h(x) - R^h(x') a^h(x)|^2~\mbox{d}x
\leq Ch^2,
\end{split}
\end{equation*}
and hence by (\ref{m3}) $\nabla V^h$ converges in 
$L^2(\Omega,\mathbb{R}^{3\times 2})$ to $A$. Consequently, by (\ref{m22}):
\begin{equation}\label{m6}
\lim_{h\to 0} V^h = V  \mbox{ in } W^{1,2}(\Omega, \mathbb{R}^{3}),
\quad V\in  W^{2,2}(\Omega, \mathbb{R}^{3})\quad  \mbox{ and } 
~~\nabla V= A_{3\times 2}.
\end{equation}
Use now (\ref{m4}) to conclude that $\mbox{sym}\nabla (V_{tan}) = 0$ 
and so by Korn's inequality $V_{tan}$ must be constant, 
hence $0$ in view of (\ref{m22}). This ends the proof of (ii).

To deduce (iii), divide both sides of (\ref{m55}) by $h$ and pass to the limit
with its symmetric part. By (\ref{m5}), the first bound in (\ref{m1}) 
and (\ref{m22}) we conclude that:
$$\|h^{-1} V^h_{tan}\|_{W^{1,2}(\Omega)}\leq 
C\|\nabla (h^{-1}V^h_{tan})\|_{L^2(\Omega)}=
C \|\mbox{sym}\nabla (h^{-1}V^h_{tan})\|_{L^2(\Omega)}\leq C,$$
which proves the claim.
\endproof

\bigskip

\noindent {\bf Proof of Theorem \ref{liminf}}

\noindent {\bf 1.} Define the rescaled strains 
$G^h\in L^2(\Omega^1, \mathbb{R}^{3\times 3})$ by:
$$G^h(x', x_3) = \frac{1}{h^2}\Big((R^h(x'))^T \nabla u^h(x', hx_3)a^h(x', hx_3)^{-1} 
- \mbox{Id}\Big).$$
Clearly, by (\ref{m1}) $\|G^h\|_{L^2(\Omega^1)}\leq C$ and hence, up to a subsequence:
\begin{equation}\label{ma0}
\lim_{h\to 0} G^h = G \qquad \mbox{weakly in } L^2(\Omega^1,\mathbb{R}^{3\times 3}).
\end{equation}
We shall now derive a property of the limiting strain $G$.
Observe first that:
\begin{equation}\label{ma1}
\lim_{h\to 0} \frac{1}{h^2} (\partial_3y^h - he_3) = Ae_3 \quad \mbox{ in }
L^2(\Omega^1,\mathbb{R}^3),
\end{equation}
where $e_3 = (0,0,1)^T$. This is because:
\begin{equation*}
\begin{split}
\frac{1}{h^2} (\partial_3&y^h(x) - he_3) 
= \frac{1}{h}\Big((\bar R^h)^T\nabla u^h(x', hx_3)
-\mbox{Id}\Big)e_3 \\ &=  \frac{1}{h}(\bar R^h)^T\Big(\nabla u^h(x', hx_3) 
- R^h(x') a^h(x', hx_3)\Big)e_3\\
& \qquad +  \frac{1}{h}(\bar R^h)^TR^h(x')\Big(a^h(x', hx_3) -\mbox{Id}\Big)e_3
+  \frac{1}{h}\Big((\bar R^h)^TR^h(x') -\mbox{Id}\Big)e_3, 
\end{split}
\end{equation*}
where the first term in the right hand side converges to $0$ in $L^2(\Omega^1)$
by (\ref{m1}), the second term to $0$ in $L^\infty(\Omega^1)$,
and the last term to $Ae_3$ in $L^2(\Omega)$ by (\ref{m3}).

For each small $s>0$ define now, with a small abuse of notation, 
the sequence of functions $f^{s,h}\in W^{1,2}(\Omega^1,\mathbb{R}^3)$:
\begin{equation}\label{ma2}
f^{s,h}(x) = \frac{1}{h^2}\frac{1}{s} \Big(y^h(x+se_3) - y^h(x) - hse_3\Big).
\end{equation}
Clearly $f^{s,h}(x) = \frac{1}{h^2}\fint_0^s y^h(x+te_3)-he_3~\mbox{d}t$, and so
by (\ref{ma1}):
$$\lim_{h\to 0} f^{s,h} = Ae_3 \quad \mbox{ in } L^2(\Omega^1, \mathbb{R}^3).$$
Also $\partial_3 f^{s,h}(x) = \frac{1}{s}\frac{1}{h^2} (\partial_3 y^h(x+se_3) 
- \partial_3 y^h(x))$ so again by (\ref{ma1}):
$$\lim_{h\to 0}\partial_3f^{s,h} = 0\quad\mbox{ in } L^2(\Omega^1,\mathbb{R}^3).$$
Further, for any $\alpha=1,2$ we have:
\begin{equation*}
\begin{split}
\partial_\alpha f^{s,h}(x) &
= \frac{1}{h^2}\frac{1}{s}(\bar R^h)^T\Big(\nabla u^h(x', hx_3+hs)
-\nabla u^h(x',hx_3)\Big)e_\alpha \\ &=  \frac{1}{s}(\bar R^h)^TR^h(x') 
\Bigg(G^h(x', x_3+s)a^h(x', hx_3+hs) - G^h(x',x_3)a^h(x',hx_3) \\
& \qquad\qquad\qquad\qquad\qquad\qquad
+ \frac{1}{h^2}\Big(a^h(x', hx_3+hs) - a^h(x', hx_3)\Big)\Bigg)e_\alpha,
\end{split}
\end{equation*}
which, in view of (\ref{ma0}) and (\ref{m2}) yields
the weak convergence in  $L^2(\Omega^1,\mathbb{R}^{3\times 2})$ of:
\begin{equation}\label{ma3}
\lim_{h\to 0} \partial_\alpha f^{s,h}(x) 
= \frac{1}{s} \Big(G(x',x_3+s) - G(x', x_3)\Big)e_\alpha
+\kappa_g(x') e_\alpha.
\end{equation}
Consequently, we see that $f^{s,h}$ converges weakly in 
$W^{1,2}(\Omega,\mathbb{R}^3)$ to $Ae_3$. 
Hence, the left hand side in (\ref{ma3}) equals $\partial_\alpha(A e_3)$ and so:
\begin{equation}\label{ma4}
G(x)_{3\times 2} = G_0(x')_{3\times 2} + x_3 G_1(x')_{3\times 2},
\end{equation}
for some $G_0\in L^2(\Omega,\mathbb{R}^{3\times 3})$ where:
\begin{equation}\label{madefG1}
G_1(x') = \nabla (A(x')e_3) - \kappa_g(x').
\end{equation}

{\bf 2.} Divide now both sides of (\ref{m55}) by $h$ and pass to the weak limit
in $L^2(\Omega,\mathbb{R}^{3\times 2})$ with its symmetric part.
Since $\lim h^{-1} \mbox{sym}\nabla V^h = \mbox{sym}\nabla w$ by Theorem 
\ref{compactness} (iii), and
$\lim h^{-1}\mbox{sym } A^h_{2\times 2} = (\mbox{sym } \epsilon_g)_{2\times 2}
+ \frac{1}{2} (A^2)_{2\times 2}$ by (\ref{m5}), and:
$$\lim_{h\to 0} \mbox{sym} \left((\bar R^h)^TR^h(x') \int_{-1/2}^{1/2} G^h(x', t)
a^h(x',ht)~\mbox{d}t\right)_{3\times 2} = G_0(x')_{3\times 2},$$
by (\ref{ma0}), (\ref{ma4}) and (\ref{m2}), we obtain:
\begin{equation}\label{ma5}
\mbox{sym} \nabla w - (\mbox{sym } \epsilon_g)_{2\times 2} 
- \frac{1}{2} (A^2)_{2\times 2}= -G_0(x')_{2\times 2}.
\end{equation}

{\bf 3.} We shall now prove the bound of the Theorem \ref{liminf}. 
First, Taylor expanding the function $W(F)$ close to $F=\mbox{Id}$, and recalling
(\ref{defQ}) we obtain:
\begin{equation*}
\begin{split}
\frac{1}{h^4} W \Big(\nabla u^h(x) a^h(x)^{-1} \Big)&
= \frac{1}{h^4} W\Big(R^h(x)^T \nabla u^h(x) a^h(x)^{-1} \Big) \\ & = 
\frac{1}{h^4} W(\mbox{Id} + h^2 G^h(x)) = \frac{1}{2} \mathcal{Q}_3(G^h(x))
+ h^2\mathcal{O}(|G^h(x)|^3).
\end{split}
\end{equation*}
Consider now sets $\Omega_h = \{x\in \Omega^1; ~~h|G^h(x', h_3)|\leq 1\}$.
Clearly $\chi_{\Omega^h}$ converges to $1$ in $L^1(\Omega^1)$, with $h\to 0$, 
as $hG^h$ converges to $0$ pointwise a.e. by (\ref{m1}).
We get:
\begin{equation}\label{ma6}
\begin{split}
\liminf_{h\to 0} \frac{1}{h^4}& I^h_W(u^h)  \geq \liminf_{h\to 0} \frac{1}{h^4}
\int_{\Omega^1}\chi_{\Omega_h}W\Big(\nabla u^h(x',hx_3)
a^h(x',hx_3)^{-1}\Big)~\mbox{d}x\\ &
= \liminf_{h\to 0} \left(\frac{1}{2}\int_{\Omega^1}\mathcal{Q}_3(\chi_{\Omega_h}G^h)
+ o(1) \int_{\Omega^1}|G^h|^2\right)\\ &
\geq \frac{1}{2}\int_{\Omega^1}\mathcal{Q}_3(\mbox{sym }G(x))~\mbox{d}x,
\end{split}
\end{equation}
where the last inequality follows by (\ref{m1}) guaranteeing convergence to $0$
of the term $o(1)\int |G^h|^2$, and by the fact that $\chi_{\Omega_h}G^h$
converges weakly to $G$ in $L^2(\Omega^1,\mathbb{R}^{3\times 3})$ (see (\ref{ma0}))
in view of $\mathcal{Q}_3$ being positive definite on and depending only 
on the symmetric part of its argument.

Further:
\begin{equation}\label{ma7}
\begin{split}
\frac{1}{2}\int_{\Omega^1}&\mathcal{Q}_3(\mbox{sym }G)  \geq
\frac{1}{2}\int_{\Omega^1}\mathcal{Q}_2(\mbox{sym }G_{2\times 2}(x))~\mbox{d}x\\ &
= \frac{1}{2}\int_{\Omega^1}\mathcal{Q}_2\Big(\mbox{sym }G_0(x')_{2\times 2}
+ x_3 \mbox{sym }G_1(x')_{2\times 2}\Big)~\mbox{d}x \\ & = 
\frac{1}{2}\int_{\Omega^1}\mathcal{Q}_2(\mbox{sym }G_0(x')_{2\times 2})
+ \frac{1}{2}\int_{\Omega^1}x_3^2\mathcal{Q}_2(\mbox{sym }G_1(x')_{2\times 2})\\ &
= \frac{1}{2}\int_{\Omega}\mathcal{Q}_2\left(\mbox{sym }\nabla w
- (\mbox{sym }\epsilon_g)_{2\times 2} - \frac{1}{2}(A^2)_{2\times 2}\right) 
\\ & \qquad\qquad\qquad
+ \frac{1}{24}\int_{\Omega}\mathcal{Q}_2\Big(\mbox{sym }(\nabla Ae_3)_{2\times 2}
- (\mbox{sym }\kappa_g)_{2\times 2}\Big),
\end{split}
\end{equation}
by (\ref{ma4}), (\ref{ma5}) and (\ref{madefG1}).
Now,  in view of Theorem
\ref{compactness} (ii) and (\ref{m6})
one easily sees that:
$$ (A^2)_{2\times 2} = -\nabla v\otimes \nabla v \quad \mbox{ and } \quad
\Big(\nabla Ae_3\Big)_{2\times 2} = -\nabla v^2,$$
which yields the claim by (\ref{ma6}) and (\ref{ma7}). 
\endproof

\section{Recovery sequence: 
a proof of Theorem \ref{thmaindue}}

For any $F \in{\mathbb R}^{2\times2}$,  by $(F)^*\in{\mathbb R}^{3\times 3}$ 
we denote the matrix for which $ (F)^*_{2\times2} = F$ 
and $(F)^*_{i3}= (F)^*_{3i} =0$, $i=1..3$.   
Recalling (\ref{defQ}), let $c(F)\in {\mathbb R}^3$ be the unique vector so that: 
$${\mathcal Q}_2 (F) = {\mathcal Q}_3 \Big( (F)^* +  
\mbox{sym}(c \otimes e_3) \Big).$$ 
The mapping $c:{\mathbb R}^{2\times 2}_{sym} \longrightarrow {\mathbb R}^3$ 
is well-defined and linear, as $\mathcal {Q}_3$ is a quadratic form, 
positive definite on the space of symmetric matrices. 
Also, for all $F\in {\mathbb R}^{3\times3}$, by $l(F)$ we denote 
the unique vector in ${\mathbb R}^3$, linearly depending on $F$,  for which:
$$\mbox{sym}\big(F - (F_{2\times 2})^*\big) 
= \mbox{sym}\big(l(F) \otimes e_3\big).$$

{\bf 1.} Let the in-plane displacement $w$ and the out-of-plane displacement
$v$ be as in Theorem \ref{thmaindue}.  
We first prove the result under the additional assumption of
$w$ and $v$ being smooth up to the boundary. 
Define the recovery sequence:
\begin{equation}\label{recoveryseq}
u^h (x', x_3)= \left[\begin{array}{c}x'\\0 \end{array}\right] 
+ \left[\begin{array}{c}h^2w(x')\\ hv(x')\end{array}\right] 
+ x_3 \left[\begin{array}{c}-h\nabla v(x')\\1\end{array}\right] 
+  h^2 x_3 d^{0}(x') + \frac{1}{2}h x_3^2 d^{1}(x),
\end{equation} 
where the smooth warping fields $d^0, d^1:\overline{\Omega}\longrightarrow
\mathbb{R}^3$ are given by: 
\begin{equation}\label{d01}
\begin{split}
d^0&=  l(\epsilon_g) - \frac 12 |\nabla v|^2 e_3  
+  c \Big(\mbox{sym} \nabla w  -  \frac 12 \nabla v \otimes \nabla v 
- (\mbox{sym } \epsilon_g)_{2\times2} \Big),  \\
d^1 &= l(\kappa_g) + c\Big(-\nabla^2 v - (\mbox{sym }\kappa_g)_{2\times 2}\Big).  
\end{split}
\end{equation}
The convergence statements in (i), (ii) and (iii) of Theorem \ref{thmaindue}
are verified by a straightforward calculation. 
In order to establish (iv) we need to estimate 
the energy of the sequence $u^h$. Calculating the deformation 
gradient we first obtain:
$$\nabla u^h = \mbox{Id} + h^2 (\nabla w)^* + hA-hx_3 (\nabla^2 v)^* 
+ h^2 \left[\begin{array}{cc} x_3 \nabla d^{0} & d^{0} \end{array} \right] 
+ h \left[\begin{array}{cc} \frac{1}{2} x_3^2 \nabla d^{1} & x_3 d^{1} 
\end{array} \right], $$
where the skew-symmetric matrx field $A$ is given as: 
$$ A = \left [\begin{array}{cc} 0 & - (\nabla v)^T\\
\nabla v & 0  \end{array}\right]. $$ 
We shall use an auxiliary $SO(3)$-valued matrix field $R^h = e^{hA}$.
Clearly: $R^h = \mbox{Id} + h A + \frac {h^2}{2} A^2 +  \mathcal{O}(h^3)$ and
$(R^h)^T = \mbox{Id} - hA + \frac {h^2}{2} A^2 + \mathcal{O}(h^3)$. 
Also, recall that:
$(a^h)^{-1} = \mbox{Id}- h^2 \epsilon_g - hx_3\kappa_g + \mathcal{O}(h^3)$.
We hence obtain:
$$(R^h)^T (\nabla u^h) (a^h)^{-1} = \mbox{Id} +  h^2\big((\nabla w)^* 
- \frac 12 A^2 - \epsilon_g+  d^0 \otimes e_3\big)
+ hx_3\big(-(\nabla^2 v)^* - \kappa_g + d^1\otimes e_3\big) + \mathcal{O}(h^3).$$  
Recalling now the definition of the quadratic form:
$\mathcal {Q}_3 (F)= D^2W(\mbox{Id})(F\otimes F)
= \mathcal Q_3 (\mbox{sym} F)$, Taylor expanding the energy density $W$
around the identity, and taking into account the uniform boundedness of all 
the involved functions and their derivatives we get:  
\begin{equation*}
\begin{split}
I^h_W(u^h) & = \frac 1h \int_{\Omega^h} W(\nabla u^h (a^h)^{-1}) 
= \frac 1h \int_{\Omega^h} W \Big((R^h)^T (\nabla u^h)(a^h)^{-1}\Big) \\ 
& = \frac {h^4}{2} \int_{\Omega} \mathcal{Q}_3\left(\mbox{sym} \Big(
(\nabla w)^* - \frac 12 A^2 - \epsilon_g+  d^0\otimes e_3 \Big)\right)\\ 
& \qquad + \frac {h^4}{24}\int_{\Omega}\mathcal {Q}_3\left(\mbox{sym}
\Big(-(\nabla^2 v)^* - \kappa_g + d^1 \otimes e_3 \Big) \right)+ \mathcal{O}(h^5). 
\end{split}
\end{equation*} 
Note that $A^2 = (\nabla v \otimes \nabla v)^* - |\nabla v|^2 (e_3 \otimes e_3)$.
Therefore:
\begin{equation*}
\begin{split}
&\mbox{sym} \left((\nabla w)^* - \frac 12 A^2 - \epsilon_g+ d^0\otimes e_3\right)
=  \left(\mbox{sym}\nabla w  + \frac 12 \nabla v \otimes\nabla v - 
(\mbox{sym }\epsilon_g)_{2\times2}\right)^*\\  
&\qquad\qquad\qquad\qquad \qquad\qquad \qquad\qquad \qquad
+ \mbox{sym}\left(\big(d^0 - l_{\epsilon_g} + \frac 12 |\nabla v|^2 e_3\big) 
\otimes e_3\right),\\
&\mbox{sym}\left(-(\nabla^2 v)^* - \kappa_g + d^1\otimes e_3 \right) 
= \left(-\nabla^2 v - (\mbox{sym }\kappa_g)_{2\times 2}\right)^* 
+ \mbox{sym}\left((d^1 - l_{\kappa_g})\otimes e_3 \right).
\end{split}
\end{equation*}     
In view of (\ref{d01}) it follows that:
\begin{equation}\label{finalestimate}
\frac{1}{h^4} I^h_W(u^h) = \mathcal{I}_g (w,v) + \mathcal{O}(h), 
\end{equation} 
which proves the desired limit (iv) for smooth displacements $w, v$.

{\bf 2.}
In order to carry out the analysis for $w\in W^{1,2}(\Omega,{\mathbb R}^2)$ 
and $v\in W^{2,2}(\Omega,\mathbb{R})$, it suffices to suitably approximate them 
in their respective norms by smooth sequences $w_h$ and $v_h$. Define the 
sequence $u^h$ as in (\ref{recoveryseq}) using $w_h$ and $v_h$ instead of 
$w$ and $v$ for each $h$. The error $\mathcal{O}(h)$
in the final estimate (\ref{finalestimate}), contains now
an additional term $h C(w_h,v_h)$
where the quantity $C(w_h, v_h)$ depends only on the higher norms of 
$w_h$ and $v_h$. This quantity can always be controlled by a uniform constant,
by slowing down the rate of convergence of the sequences $w_h$ and $v_h$.
\endproof

\section{Euler-Lagrange equations of the functional $\mathcal{I}_g$:
the derivation of (\ref{EL-2}) and the free boundary conditions}

Assume that $(w,v)$ is a local minimizer of $\mathcal{I}_g$ 
in (\ref{vonKarman}) with $\mathcal{Q}_2$ as in (\ref{Q23}).

{\bf 1.} Consider a variation $\phi\in\mathcal{C}_0^\infty(\Omega, \mathbb{R}^2)$
in $w$. That is, for all small (positive or negative) $\varepsilon$:
$$ \mathcal{I}_g(w+\varepsilon \phi, v) - \mathcal{I}_g(w, v) \geq 0.$$
Collecting terms of order $\varepsilon$, we obtain that:
\begin{equation}\label{7.0}
\int_\Omega \left(2\mu (\mbox{sym }\nabla\phi) : (\mbox{sym } \nabla w +\Psi)
+ \frac{2\mu\lambda}{2\mu+\lambda}(\mbox{div}\phi)~ 
(\mbox{div }w + \mbox{tr }\Psi)\right) = 0,
\end{equation}
where $\Psi\in  L^2(\Omega,\mathbb{R}^{2\times 2})$ is the vector field:
\begin{equation}\label{Psi}
\Psi = \frac{1}{2} \nabla v\otimes\nabla v 
- (\mbox{sym } \epsilon_g)_{2\times 2}.
\end{equation}
After integrating (\ref{7.0}) by parts and recalling the fundamental theorem of 
calculus of variations, we obtain:
\begin{equation}\label{el1}
\mbox{div } \mathbf{M} = 0, \qquad \mbox{ with } \quad
\mathbf{M}=2\mu (\mbox{sym }\nabla w + \Psi) 
+ \frac{2\mu\lambda}{2\mu+\lambda}(\mbox{div }w + \mbox{tr }\Psi)\mbox{Id}.
\end{equation}
Above, the divergence of a symmetric matrix field $\mathbf{M}$ is taken row-wise.
Consequently, the $i$-th row of $\mathbf{M}$ ($i=1,2$) 
can be written as $\nabla^\perp\psi^i$ for some scalar 
fields $(\psi^1,\psi^2)$. The symmetric matrix field $\mbox{cof }\mathbf{M}$ 
has hence the form $\nabla (\psi^2,-\psi^1)^T$, which implies that:
$$\mbox{cof } \mathbf{M} = \nabla^2\Phi \quad \mbox{ for some } \quad
\Phi\in W^{2,2}(\Omega,\mathbb{R}).$$
Recall that for a matrix $M\in\mathbb{R}^{n\times n}$, $\mbox{cof } M$ 
denotes the matrix of cofactors of $M$, that is 
$(\mbox{cof } M)_{ij} = (-1)^{i+j} \det \hat M_{ij}$, where 
$\hat M_{ij}\in\mathbb{R}^{(n-1)\times (n-1)}$ is obtained from $M$ by deleting 
its $i$th row and $j$th column.

From this discussion we see that (\ref{el1}) is equivalent with:
\begin{equation}\label{el2}
\mathbf{M} = \mbox{cof }\nabla^2\Phi. 
\end{equation}
In classical elasticity, the scalar field $\Phi$ is called the Airy stress potential.

{\bf 2.} We shall now need the following result:
\begin{lemma}\label{lemcurltcurl}
Let $\alpha,\beta\in\mathbb{R}$ be such that: $\alpha\neq 0$ 
and $\alpha+2\beta\neq 0$.
Then the following conditions are equivalent, for any matrix 
field $F\in L^2(\Omega,\mathbb{R}^{2\times 2})$:
\begin{itemize}
\item[(i)] $F=\alpha~\mathrm{ sym}\nabla w + \beta (\mathrm{div }~w)\mathrm{Id}$,
for some $w\in W^{1,2}(\Omega,\mathbb{R}^2)$,
\item[(ii)] $\displaystyle \mathrm{curl}^T\mathrm{curl }~ B
- \frac{\beta}{\alpha+2\beta} \Delta (\mathrm{tr }F) = 0$,
in the sense of distributions. 
\end{itemize}
\end{lemma}
\begin{proof}
We shall use the following easily obtained formulas:
$$\mbox{curl}^T\mbox{curl } (\gamma\mbox{Id}) = \Delta\gamma=
\mbox{div}^T\mbox{div } (\gamma\mbox{Id}),$$
valid for any scalar field $\gamma$ on $\Omega$.

To prove the implication (i)$\Rightarrow$(ii), note that by (i):
$\mbox{tr }F = (\alpha+2\beta)\mbox{ div }w$. Thus:
$$\mbox{curl}^T\mbox{curl } F = \beta~\mbox{curl}^T\mbox{curl }\Big((\mbox{div }w)
\mbox{Id}\Big) = \beta~ \Delta (\mbox{div } w) = \frac{\beta}{\alpha+2\beta}~
\Delta (\mathrm{tr }F). $$

To prove the reverse implication (ii)$\Rightarrow$(i) observe that by (ii):
$$\mbox{curl}^T\mbox{curl } F - \frac{\beta}{\alpha+2\beta} 
\mbox{curl}^T\mbox{curl } \Big((\mbox{tr }F)\mbox{Id}\Big)=0.$$ 
Thus:
\begin{equation}\label{elpomoc}
F - \frac{\beta}{\alpha+2\beta} (\mbox{tr }F)\mbox{Id} = \alpha ~\mbox{sym }\nabla w,
\end{equation}
for some vector field $w$. In particular:
$$ \mbox{div } w = \mbox{tr }(\mbox{sym }\nabla w)=
\frac{1}{\alpha} \left(\mbox{tr }F - \frac{2\beta}{\alpha+2\beta}~\mbox{tr }F\right)
= \frac{1}{\alpha+2\beta}\mbox{tr }F.$$
Together with (\ref{elpomoc}) the above implies (ii).
\end{proof}

\medskip

We now use Lemma \ref{lemcurltcurl} with: 
$$F=\mbox{cof }\nabla^2\Phi - \left(2\mu \Psi
+ \frac{2\mu\lambda}{2\mu + \lambda}(\mbox{tr }\Psi)\mbox{Id}\right),
\quad \alpha=2\mu, \quad\beta= 2\mu\lambda/(2\mu+\lambda).$$
Clearly, the condition (i) is equivalent to  (\ref{el2}) and hence 
(\ref{el2}) is further equivalent to (ii), which after recalling (\ref{Psi})
takes the form:
\begin{equation*}
\begin{split}
\mbox{curl}^T\mbox{curl }\Big(\mbox{cof }\nabla^2\Phi\Big) 
- \frac{\lambda}{2\mu+3\lambda}~
\Delta (\mbox{tr }\mbox{cof }\nabla^2\Phi) = \mu~ \mbox{curl}^T\mbox{curl }\Big(
\nabla v\otimes\nabla v - 2(\mbox{sym }\epsilon_g)_{2\times 2}\Big).
\end{split}
\end{equation*}
Since $\mbox{curl}^T\mbox{curl }(\mbox{sym } \epsilon_g)_{2\times 2} 
= \mbox{curl}^T\mbox{curl }(\epsilon_g)_{2\times 2}$, and
$\mbox{curl}^T\mbox{curl }(\nabla v\otimes\nabla v) = -2~\mbox{det }\nabla^2v$
and both $\mbox{curl}^T\mbox{curl }(\mbox{cof }\nabla^2\Phi)$, and 
$\Delta (\mbox{tr }\mbox{cof }\nabla^2\Phi)$ equal $\Delta^2\Phi$, we obtain:
\begin{equation*}
\frac{2(\mu+\lambda)}{2\mu+3\lambda}~ \Delta^2\Phi 
= -2\mu\Big(\mbox{det}\nabla^2v 
+ \mbox{curl}^T\mbox{curl}(\epsilon_g)_{2\times 2}\Big),
\end{equation*}
or equivalently:
\begin{equation}\label{el3}
\Delta^2\Phi 
= -S\Big(\mbox{det}\nabla^2v 
+ \mbox{curl}^T\mbox{curl}(\epsilon_g)_{2\times 2}\Big).
\end{equation}

{\bf 3.}  Consider now a variation 
$\varphi\in\mathcal{C}_0^\infty(\Omega, \mathbb{R})$
in $v$, so that for all small positive and negative $\varepsilon$:
$$ \mathcal{I}_g(w, v+\varepsilon \varphi) - \mathcal{I}_g(w, v) \geq 0.$$
Collecting terms of order $\varepsilon$
and calling: 
\begin{equation}\label{7.5}
\tilde\Psi = \nabla^2v + (\mbox{sym }\kappa_g)_{2\times 2},
\end{equation}
we obtain that:
\begin{equation}\label{el4}
\int_\Omega\Big( (\nabla\varphi\otimes\nabla v):\mathbf{M} + B
\nabla^2\varphi : (\tilde\Psi + \nu \mbox{ cof }\tilde\Psi)\Big) = 0,
\end{equation}
where we used the following identity, valid for any $F\in \mathbb{R}^{2\times 2}_{sym}$:
$$2\mu F  + \frac{2\mu\lambda}{2\mu+\lambda}~(\mbox{tr }F)\mbox{Id}
= 12 B (F + \nu \mbox{ cof } F).$$
By (\ref{el2}) the first term in the integrand of (\ref{el4}) equals
$\nabla \varphi \cdot ((\mbox{cof }\nabla^2\Phi)\nabla v)$. 
Integrate by parts in (\ref{el4}) and use fundamental theorem of calculus of
variations to obtain:
\begin{equation}\label{el5}
-\mbox{div} \Big((\mbox{cof }\nabla^2\Phi)\nabla v\Big) +
B \mbox{ div}^T\mbox{div}\left(\tilde\Psi 
+ \nu \mbox{ cof }\tilde\Psi\right)=0.
\end{equation}
Use now the following formulas:
$$\mbox{div} \Big((\mbox{cof }\nabla^2\Phi)\nabla v\Big) =
(\mbox{cof }\nabla^2\Phi):\nabla^2 v,\qquad
\mbox{div}^T\mbox{div }(\nabla^2 v) = \Delta^2 v,$$ 
and remember that $\mbox{div}\mbox{cof}$ of a gradient of a vector field
vanishes, to find the following equivalent form of (\ref{el5}):
\begin{equation}\label{el6}
\begin{split}
(\mbox{cof }\nabla^2\Phi):\nabla^2 v = B \Delta^2 v + 
B \mbox{ div}^T\mbox{div}\Big((\mbox{sym }\kappa_g)_{2\times 2}
+ \nu \mbox{ cof }(\mbox{sym }\kappa_g)_{2\times 2}\Big)
\end{split}
\end{equation}
Recalling the definition of Airy' bracket:
$$[v,\Phi] = \nabla^2 v : (\mbox{cof }\nabla^2\Phi),$$ 
so that $\mbox{det }\nabla^2 v = 1/2 [v,v]$, we see that (\ref{el3}) 
and (\ref{el6}) give the system (\ref{EL-2}).

\medskip

{\bf 4.} We will now derive the natural (free) boundary conditions 
satisfied by the minimizers $(w,v)$ of $\mathcal{I}_g$ in (\ref{vonKarman}). 
The analysis is equivalent as the above, for the variations 
$\phi$ and $\varphi$ which do not vanish on the boundary of $\partial\Omega$. 

Integrating (\ref{7.0}) by parts and taking into account (\ref{el1}), we obtain:
\begin{equation}\label{7.00}
 {\mathbf M} \vec n =0 \quad \mbox{ on } \partial\Omega, 
\end{equation}
where $\vec n$ denotes the normal to $\partial\Omega$. Hence by (\ref{el2}):
$(\mbox{cof }\nabla^2\Phi)\vec n=0$, which is equivalent to:
$ \partial_\tau \nabla \Phi =0$, 
for the tangent vector field $\tau$ to $\partial\Omega$.
Therefore:
$$ \nabla \Phi \equiv const. \quad \mbox{ on } \partial \Omega. $$ 
Since $\Phi$ is detemined up to affine functions, we may assume that  
$\Phi(x_0) $ and $\nabla \Phi(x_0)$ vanish at 
a given point $x_0\in \partial \Omega$. 
We obtain hence the first set of boundary conditions, for $\Phi$: 
\begin{equation}\label{bc1}
\Phi = \partial_{\vec n} \Phi =0  \quad \mbox{ on } \partial \Omega.
\end{equation} 
To deduce the boundary conditions for the out-of-plane displacement $v$ 
we use (\ref{el4}) which is again valid 
for all $\varphi \in \mathcal{C}^\infty(\overline\Omega,\mathbb{R})$. 
Integrating by parts as before and applying (\ref{el5}) yields:
\begin{equation*}
\begin{split}
\int_{\partial \Omega} \Big(\mathbb{M} : (\nabla v \otimes \vec n)
~ \varphi + B ~\nabla\phi \cdot (\tilde\Psi +\nu 
\mbox{ cof }\tilde\Psi) \vec n 
- B \mbox{ div}(\tilde\Psi +\nu \mbox{ cof }\tilde\Psi)
\cdot\vec n~\varphi\Big) =0.  
\end{split}
\end{equation*} 
The first term above drops out by (\ref{7.00}).
Writing $\nabla\varphi = (\partial_\tau\varphi)\tau 
+ (\partial_{\vec n}\varphi)\vec n$ 
we obtain the following new boundary equations:
\begin{equation*}
\left(\tilde\Psi +\nu \mbox{ cof }\tilde\Psi\right) 
: (\vec n\otimes\vec n) = 0.
\end{equation*}
and:
\begin{equation*}
\partial_\tau\Big((\tilde\Psi +\nu 
\mbox{ cof }\tilde\Psi):(\vec n\otimes \vec\tau)\Big)  
+ \mbox{div} \left(\tilde\Psi +\nu 
\mbox{ cof }\tilde\Psi\right) \vec n = 0,
\end{equation*}
which are, respectively, equivalent to:
\begin{equation}\label{b1}
\tilde\Psi : (\vec n\otimes\vec n) + \nu~\tilde\Psi : (\tau\otimes\tau)  = 0
\qquad \mbox{on } \partial\Omega,
\end{equation}
\begin{equation}\label{b2}
(1 - \nu) \partial_\tau\Big(\tilde\Psi : (\vec n\otimes\tau)\Big) 
+ \mbox{div} \left(\tilde \Psi + \nu~\mbox{cof }\tilde \Psi\right)\vec n =0
\qquad \mbox{on } \partial\Omega. 
\end{equation}
In the particular case when $(\mbox{sym }\kappa_g)_{2\times 2}=0$ 
on $\partial\Omega$, (\ref{b1}) and (\ref{b2}) become:
\begin{equation*}
\begin{split}
&\partial^2_{\vec n \vec n} v 
+ \nu\Big(\partial^2_{\tau\tau} v - K\partial_{\vec n} v\Big) = 0\\
&(2-\nu) \partial_{\tau} \partial_{\vec n} \partial_{\tau} v +
\partial^3_{\vec n \vec n \vec n} v
+ K\Big(\Delta v + 2\partial^2_{\vec n \vec n} v\Big) = 0,
\end{split}
\end{equation*}
where $K$ stands for the (scalar) curvature of $\partial\Omega$, so that
$\partial_{\tau}\tau = K\vec n$.
If additionally $\partial\Omega$ is a polygonal, then
the above equations simplify to equations (5) in \cite{Maha}.
\endproof

\section{Discussion 
and a proof of Corollary \ref{minsconverge}}

Recall that in classical elasticity it is usually the magnitude of the applied
forces or the types of boundary conditions which
determine  the behavior of thin plates or shells, see e.g. \cite{FJMhier}. 
Such exterior constraints are replaced in our case by the geometric constraints 
induced by the prescribed metric. In this line, we  conjecture existence of
a hierarchy of limit model theories,  
depending on the choice of the tensor $a^h$ whose
qualifications can be predicted by geometric observations.
Indeed, in what follows, we demonstrate how $a^h$ in (\ref{ahform})
is related to conditions (\ref{CO1}) and (\ref{CO2}) through 
the Gauss-Codazzi equations or through an expansion of the 
Riemann curvature tensor of the metric $G^h = (a^h)^T a^h$.
In particular, these two conditions can be interpreted as 
the leading order defect in $G^h$ from being a flat metric. Other
choices of scalings in $a^h$ should in turn
impose the corresponding scalings of the energy and the 
acceptable displacements for the limit model.

Corollary \ref{minsconverge} follows now from 
the next result:

\begin{lemma}\label{GCM}
For any $w\in W^{1,2}(\Omega,\mathbb{R}^2)$ and 
$v\in W^{2,2}(\Omega,\mathbb{R})$, the following are equivalent:
\begin{itemize}
\item[(i)] $\mathcal{I}_g(w,v) = 0$,
\item[(ii)] $\mathrm{curl }\big((\mathrm{sym }~\kappa_g)_{2\times 2}\big) = 0$
and $\displaystyle\mathrm{curl}^T\mathrm{curl}~ (\epsilon_g)_{2\times 2}= 
-  \mathrm{det}\big((\mathrm{sym }~\kappa_g)_{2\times 2}\big).$
\end{itemize}
The two equations in (ii) are the linearised Gauss-Codazzi-Meinardi 
equations corresponding to the metric
$\mathrm{Id} + 2 h^2 \mathrm{sym} (\epsilon_g)_{2\times 2}$
and the shape operator $h(\mathrm{sym }~\kappa_g)_{2\times 2}$
on the mid-plate $\Omega$.
\end{lemma}
\begin{proof}
Recall that for a matrix field $B\in L^2(\Omega, \mathbb R^{2\times 2}_{sym})$ 
the following two assertions hold true:
\begin{equation}\label{rule1}
\mbox{curl } B =0 \Longleftrightarrow
B = -\nabla^2 v \quad \mbox{ for some } v\in W^{2,2}(\Omega,\mathbb{R})
\end{equation} 
\begin{equation}\label{rule2}
\mathrm{curl}^T \mathrm{curl}~ B =0 \Longleftrightarrow B = \mbox{sym } \nabla w
\quad \mbox{ for some } w \in W^{1,2}(\Omega,{\mathbb R}^2).
\end{equation}
By (\ref{rule1}), the first identity in (ii) is equivalent to:
$$(\mathrm{sym}~ \kappa_g)_{2\times 2} = - \nabla^2 v.$$ 
Consequently, the second identity in (ii) becomes:
$$ \mathrm{curl}^T\mathrm{curl} \,\Big (-\frac 12 \nabla v\otimes \nabla v 
+ (\mathrm{sym} ~\epsilon_g)_{2\times 2} \Big ) =0, $$ 
which, in view of (\ref{rule2}), is equivalent to the existence of $w$ with:
$$ \mathrm{sym }\nabla w +\frac{1}{2}\nabla v\otimes \nabla v
- (\mathrm{sym}~ \epsilon_g)_{2\times 2}= 0. $$ 
Since $\mathcal {Q}_2$ is positive definite on symmetric matrices
we see that indeed (ii) is equivalent to the vanishing of both terms in
$\mathcal{I}_g$.


To identify the equations in (ii), recall the Gauss-Codazzi-Meinardi system \cite{HH}:
\begin{equation}\label{Gauss}
\begin{split}
&\partial_2 L - \partial_1 M = L \Gamma^1_{12} + M(\Gamma^2_{12} - \Gamma^1_{11}) 
- N \Gamma^2_{11},\\
&\partial_2M- \partial_1N = L\Gamma^1_{22} + M(\Gamma^2_{22} - \Gamma^1_{21}) 
- N \Gamma^2_{21},\\
&\qquad\qquad LN-M^2 = K(EG-F^2)
\end{split}
\end{equation} 
which provides the necessary and sufficient conditions for existence of 
a surface with the first and second fundamental forms:
\begin{equation*}
I= [g_{\alpha\beta}]= \left [ 
\begin{array}{cc} 
E & F \\ F & G
\end{array}
\right ] \quad \mbox{and}  \quad 
II =
 \left [ \begin{array}{cc}
L & M \\ M & N 
\end{array}
\right ].
\end{equation*} 
In (\ref{Gauss}) $\Gamma^i_{jk}$ denote the Christoffel symbols,
and $K$ stands for the Guassian curvature, which can be calculated 
from $I$ and $II$. Substituting now $I= \mathrm{Id} + 2 h^2 (\epsilon')_{2\times 2}$, 
$II = h(\kappa')_{2\times 2}$, where  $\epsilon'$ and $\kappa'$  
respectively  denote the symmetric parts of 
$\epsilon_g$ and $\kappa_g$, and taking into account the relations:
$$ \Gamma^k_{ij} = \frac 12 g^{kl} (\partial_j g_{il}+ \partial_i g_{jl} 
- \partial_l g_{ij})= 
(\partial_j \epsilon'_{ik}+ \partial_i \epsilon'_{jk} - \partial_k \epsilon'_{ij})h^2  
+ \mathcal{O}(h^2), $$ 
$$ K= \frac{R_{1212}}{EG-F^2} $$ 
we directly obtain the first identity in (ii):
\begin{equation*}
h \Big (\partial_2 (\kappa')_{11} - \partial_1 (\kappa')_{12}\Big ) = \mathcal{O}(h^2) 
\quad \mbox{ and } \quad
h \Big ( \partial_2 (\kappa')_{12}- \partial_1 (\kappa')_{22}\Big ) 
= \mathcal{O}(h^2), 
\end{equation*} 
and $h^2\mbox{det } (\kappa')_{2 \times 2}  = R_{1212}$.
Now recall that the Riemann curvatures are given by:
$$ R_{ijkl} =  g_{lm} (\partial_k \Gamma^m_{ij} - \partial_j \Gamma^m_{ik} 
+ \Gamma^n_{ij} \Gamma^m_{nk} - \Gamma^n_{ik} \Gamma^m_{nj}) 
= \partial_k \Gamma^l_{ij} - \partial_j \Gamma^l_{ik} + \mathcal{O}(h^4). $$ 
Hence, after straightforward calcuations, we obtain:
$$ h^2 \mbox{det}\, (\kappa')_{2 \times 2} = R_{1212} 
= \partial_2 \Gamma^2_{11} - \partial_1 \Gamma^2_{12} + \mathcal{O}(h^4) = 
- h^2 \Big (\mbox{curl}^T \mbox{curl} (\epsilon')_{2\times 2} \Big )+ \mathcal{O}(h^4),$$
which yields the second identity in (ii).   
\end{proof}


\section{Appendix A.
Approximating low energy deformations: a proof
of Theorem \ref{thapprox}}
The first crucial observation follows from the below rigidity estimate, 
which reproduces that of \cite{LePa1}, and it is a non-Euclidean 
version of the bound in \cite{FJMgeo}:
\begin{lemma}\label{lemrigidity}
For every $u\in W^{1,2}(\Omega^h,\mathbb{R}^3)$ and every $x_0\in\Omega$ 
there exists $R\in SO(3)$ such that:
$$\frac{1}{h}\int_{\Omega^h}|\nabla u(x) - Ra^h(x_0)|^2~\mathrm{d}x\leq
C \left(I^h_0(u) + (\mathrm{diam }~ \Omega^h)^2 Var^2(a^h) |\Omega| \right).$$
The constant $C$ above depends on 
$\|a^h\|_{L^\infty}$, $\|(a^h)^{-1}\|_{L^\infty}$, and on the domain
$\Omega^h$. Its dependence on $\Omega^h$ is uniform for a family
of plates which are bilipschitz equivalent with controlled Lipschitz constants.
\end{lemma}
\begin{proof}
Recall that according to the celebrated result in \cite{FJMgeo}, for every
$v\in W^{1,2}(\mathcal{V},\mathbb{R}^n)$ defined on an open, bounded set 
$\mathcal{V}\subset\mathbb{R}^n$, there exists $R\in SO(3)$ such that:
\begin{equation}\label{basrig} 
\int_{\mathcal{V}} |\nabla v - R|^2 \leq C_{\mathcal{V}} 
\int_{\mathcal{V}} \mbox{dist}^2(\nabla v, SO(3)). 
\end{equation}
The constant $C_{\mathcal{V}}$ depends only on the domain $\mathcal{V}$
and it is uniform for a family of domains which are bilipschitz equivalent
with controlled Lipschitz constants.

In the present setting call 
$A_0= a^h(x_0)$ and apply (\ref{basrig}) 
to the vector field $v(y) = u(A^{-1}_0y)\in W^{1,2}(A_0\Omega^h, \mathbb{R}^3)$.
After change of variables we obtain:
$$ \exists R\in SO(3) \qquad
\int_{\Omega^h} |(\nabla u) A^{-1}_0 - R|^2 \leq C_{A_0\Omega^h} 
\int_{\Omega^h} \mbox{dist}^2((\nabla u) A_0^{-1}, SO(3)). $$
Since the set $A_0\Omega^h$ is a bilipschitz image of $\Omega^h$, the constant
$C_{A_0\Omega^h}$ has a uniform bound $C$ 
depending on $|A_0|$, $|A_0^{-1}|$ and $\Omega^h$.  Further:
\begin{equation*}
\begin{split}
\frac{1}{h}\int_{\Omega^h} &|\nabla u - RA_0|^2 \leq C|A_0|^4  
h^{-1}\int_{\Omega^h} \mbox{dist}^2(\nabla u, SO(3) A_0)\\
&\leq  C |A_0|^4 h^{-1}\left(\int_{\Omega^h} 
\mbox{dist}^2\Big(\nabla u(x), SO(3) a^h(x)\Big)~\mbox{d}x
+ \int_{\Omega^h} |a^h(x) - a^h(x_0)|^2~\mbox{d}x\right)\\
&\leq C |A_0|^6\left(I_0^h(u) + h^{-1}\int_{\Omega^h} |a^h(x) - a^h(x_0)|^2\right).
\end{split}
\end{equation*}
The claim follows now through:
\begin{equation}\label{3.0}
\begin{split}
\int_{\Omega^h} |{a^h}(x) &- {a^h}(x_0)|^2~\mbox{d}x \leq
2\int_{\Omega^h}(|a^h(x) - a^h(x')|^2 + |a^h(x') - a^h(x_0)|^2)~\mbox{d}x\\
&\leq C \int_{\Omega^h}h^2 |\partial_3 a^h|^2
+ \|\nabla_{tan} (a^h_{~|\Omega})\|_{L^\infty}^2 
(\mathrm{diam } ~\Omega^h)^2~\mbox{d}x.
\end{split}
\end{equation}
\end{proof}

\noindent{\bf Proof of Theorem \ref{thapprox}}.

\noindent The proof follows the line of Theorem 10 \cite{FJMhier} 
(see also Lemma 8.1 \cite{lemopa1}).

{\bf 1.} Let $D_{x',h}=B(x',h)\cap\Omega$ be $2$d curvilinear discs
in $\Omega$ of radius $h$ and centered at a given $x'\in\Omega$. 
On each 3d plate $B_{x',h}=D_{x',h}\times(-h/2, h/2)$ 
use Lemma \ref{lemrigidity} to obtain $R_{x',h}\in SO(3)$ such that:
\begin{equation}\label{pr1}
\begin{split}
\frac{1}{h}&\int_{B_{x',h}} |\nabla u^h - R_{x',h}a^h(x)|^2 \\
&\qquad
\leq C \Bigg(h^{-1}\int_{B_{x',h}}
\mbox{dist}^2(\nabla u^h(a^h)^{-1},SO(3))~\mbox{d}x 
+ h^2 Var^2(a^h) |D_{x',h}|\Bigg)
\end{split}
\end{equation}
with a universal constant $C$, depending only on 
the Lipschitz constant of $\partial\Omega$, but independent of $h$.
Notice that we have also used (\ref{3.0}) to exchange $a^h(x')$ with $a^h(x)$
in the left hand side above.

Consider now the family of mollifiers 
$\eta_{x'}:\Omega\longrightarrow \mathbb{R}$,
parametrized by $x'\in\Omega$:
$$\eta_{x'}(z') 
= \frac{\theta(|z' - x'|/h)}{h\int_{\Omega}\theta(|y'-x'|/h)~\mbox{d}y'},$$
where $\theta\in\mathcal{C}_c^\infty([0,1))$ is a nonnegative cut-off function, 
equal to a nonzero constant in a neighborhood of $0$. Then $\eta_{x'}(z')=0$
for all $z'\not\in D_{x,h}$ and:
$$\int_\Omega \eta_{x'} = h^{-1},\quad
\|\eta_{x'}\|_{L^\infty}\leq Ch^{-3}, \quad 
\|\nabla_{x'}\eta_{x'}\|_{L^\infty}\leq Ch^{-4}.$$

{\bf 2.} Define now $Q^h\in W^{1,2}(\Omega,\mathbb{R}^{3\times 3})$:
$$ Q^h(x')= \int_{\Omega^h} \eta_{x'}(z') \nabla u^h(z)a^h(z)^{-1}~\mbox{d}z.$$
By (\ref{pr1}), we obtain the following pointwise estimates, 
for every $x'\in\Omega$:
\begin{equation}\label{3.5}
\begin{split}
|Q^h(x') &- R_{x',h}|^2 \leq \left( \int_{\Omega^h}\eta_{x'}(z') 
\left(\nabla u^h(z)a^h(z)^{-1} - R_{x',h}\right)~\mbox{d}z\right)^2 \\
&\leq \int_{\Omega^h}|\eta_{x'}(z')|^2~\mbox{d}z \cdot
\int_{B_{x',h}} |\nabla u^h (a^h)^{-1} - R_{x',h}|^2\\
& \leq Ch^{-3} \left(\int_{B_{x',h}} \mbox{dist}^2(\nabla u^h(a^h)^{-1},SO(3))
~\mbox{d}z + h^3 Var^2(a^h)|D_{x',h}|\right)
\end{split}
\end{equation}
\begin{equation*}
\begin{split}
|\nabla Q^h(x')|^2 & 
= \left(\int_{\Omega^h}(\nabla_{x'}\eta_{x'}(z'))
\left(\nabla u^h(z)a^h(z)^{-1} - R_{x',h}\right)~\mbox{d}z\right)^2\\
&\leq \int_{B_{x',h}}|\nabla_{x'}\eta_{x'}(z')|^2~\mbox{d}z \cdot
\int_{B_{x',h}} |\nabla u^h(a^h)^{-1} - R_{x',h}|^2\\
&\leq C h^{-5} \left(\int_{B_{x',h}} \mbox{dist}^2(\nabla u^h (a^h)^{-1},SO(3))
~\mbox{d}z + h^3 Var^2(a^h)|D_{x',h}|\right).
\end{split}
\end{equation*}
Applying the estimates above and in (\ref{pr1}) on doubled balls $B_{x', 2h}$ 
we arrive at:
\begin{equation*}
\begin{split}
\frac{1}{h}\int_{B_{x',h}}&|\nabla u^h(z)a^h(z)^{-1} - Q^h(z')|^2~\mbox{d}z \\
&\leq 
C \left( \frac{1}{h}\int_{B_{x',h}}|\nabla u^h(a^h)^{}-1 - R_{x',2h}|^2
+ \frac{1}{h}\int_{B_{x',2h}}|Q^h(z') - R_{x',2h}|^2~\mbox{d}z\right)\\
&\leq C \left(h^{-1}\int_{B_{x',2h}} \mbox{dist}^2(\nabla u^h(a^h)^{-1},SO(3))
~\mbox{d}z + h^2 Var^2(a^h)|D_{x',2h}|\right),
\end{split}
\end{equation*}
\begin{equation*}
\int_{D_{x',h}}|\nabla Q^h|^2 \leq 
Ch^{-2} \left(h^{-1}\int_{B_{x',2h}} \mbox{dist}^2(\nabla u^h(a^h)^{-1},SO(3))
~\mbox{d}z + h^2 Var^2(a^h)|D_{x',2h}|\right).
\end{equation*}
Consider a finite covering $\Omega = \bigcup D_{x',h}$  
whose intersection number is independent of $h$ (as it depends only on the 
Lipschitz constant of $\partial\Omega$). 
Sum the above bounds:
\begin{equation}\label{pr2} 
\frac{1}{h}\int_{\Omega^h}|\nabla u^h(x) - Q^h(x')a^h(x)|^2~\mathrm{d}x
\leq C \left(I_0^h(u^h) + h^2 Var^2(a^h)\right),
\end{equation}
\begin{equation}\label{pr3} 
\int_\Omega |\nabla Q^h|^2 \leq C h^{-2} \left(I_0^h(u^h) + h^2 Var^2(a^h)\right)
\end{equation}

{\bf 3.} Notice that by (\ref{3.5}):
$$\mbox{dist}^2(Q^h(x'),SO(3)) \leq |Q^h(x') - R_{x',h}|^2
\leq C\left(h^{-2}I_0^h(u^h) + h^2 Var^2(a^h)\right) 
\rightarrow 0 \quad \mbox{as } h\to 0, $$
in view of assumption $(\spadesuit)$. We may therefore, 
for small $h$, project $Q^h$ onto $SO(3)$:
$$R^h(x') = \mathbb{P}_{SO(3)}(Q^h(x')).$$
We further have, by (\ref{pr2}):
\begin{equation*}
\begin{split}
\int_\Omega& |R^h(x') - Q^h(x')|^2~\mbox{d}x' =
\int_\Omega \mbox{dist}^2(Q^h(x'),SO(3))
~\mbox{d}x' \\
&\leq C h^{-1}\Bigg(\int_{\Omega^h}|\nabla u^h(x)a^h(x)^{-1} - Q^h(x')|^2~\mbox{d}x
+ \int_{\Omega^h}\mbox{dist}^2(\nabla u^h(x)a^h(x)^{-1},SO(3))~\mbox{d}x\Bigg)\\
&\leq C\left(I_0^h(u^h) + h^2Var^2(a^h)\right).
\end{split}
\end{equation*}
On the other hand $|\nabla R^h|\leq C |\nabla Q^h|$ and the claim
follows by (\ref{pr2}) and (\ref{pr3}).
\endproof


\section{Appendix B: The $\Gamma$-convergence formalism}

Theorems \ref{liminf} and \ref{thmaindue} can be summarized
using the language of $\Gamma$-convergence \cite{dalmaso}.
Recall  that a sequence of functionals
$\mathcal{F}^h:X\longrightarrow \overline{\mathbb{R}}$ defined on a metric
space $X$, is said to $\Gamma$-converge, as $h\to 0$, to
$\mathcal{F}:X\longrightarrow \overline{\mathbb{R}}$ provided that
the following two conditions hold:
\begin{itemize}
\item[(i)] For any converging sequence $\{x^h\}$ in $X$:
\begin{equation*}\label{Gamma1}
\mathcal{F}\left(\lim_{h\to 0} x^h\right) \leq \liminf_{h\to 0}
\mathcal{F}^h(x^h).
\end{equation*}
\item[(ii)] For every $x\in X$, there exists a sequence $\{x^h\}$ converging
to $x$ and such that:
\begin{equation*}\label{Gamma2}
\mathcal{F}(x) = \lim_{h\to 0} \mathcal{F}^h(x^h).
\end{equation*}
\end{itemize}

\begin{corollary}
Define the sequence of functionals:
\begin{equation*} 
\begin{split}
&\mathcal{F}^h: W^{1,2}(\Omega^1,\mathbb{R}^3)
\times W^{1,2}(\Omega,\mathbb{R}^3)\times W^{1,2}(\Omega, \mathbb{R}^2)
\longrightarrow \overline{\mathbb{R}}\\
&\mathcal{F}^h(y, V, w) = \left\{\begin{array}{ll}
\displaystyle{\frac{1}{h^4}I^h_W(y(x',hx_3))} & 
\mbox{ if }~ V(x') = \fint y(x',t) - x'~\mathrm{d}t
\mbox{ and }~ w = h^{-1}V_{tan},\\
+\infty & \mbox{ otherwise.}
\end{array}\right.
\end{split}
\end{equation*}
Then $\mathcal{F}^h$ $\Gamma$-converge, as $h\to 0$, to the following functional:
\begin{equation*} 
\mathcal{F}(y, V, w) = \left\{\begin{array}{ll}
\mathcal{I}_g(w,v) & ~\mathrm{ if }~ y(x',t) = x'\mbox{ and } ~V=(0,0,v)^T\in W^{2,2},\\
+\infty & \mbox{ otherwise.}
\end{array}\right.
\end{equation*}
Consequently, the (global) approximate minimizers of $\mathcal{F}^h$
converge to a global minimizer of $\mathcal{F}$.
\end{corollary}

\noindent{\bf Acknowledgments.}
M.L. was partially supported by the NSF grants DMS-0707275 and DMS-0846996 and by the Center for Nonlinear Analysis (CNA) under
the NSF grants 0405343 and 0635983. L.M. was partially supported by the Harvard NSF-MRSEC and the MacArthur Foundation. R.P. was partially 
supported by the NSF grant DMS-0907844.

\end{document}